\documentclass[12pt]{article}

\usepackage{amsmath, amsthm, amssymb, accents,amscd}
\usepackage{enumerate}
\usepackage{pdflscape}
\usepackage{caption}

\usepackage{ifpdf}
\ifpdf
\usepackage[pdftex]{graphicx}
\else
\usepackage[dvips]{graphicx}
\fi
\usepackage{tikz}
 	 \usetikzlibrary{arrows,backgrounds,math}
\usepackage[all]{xy}
\usepackage{tikz-cd}
\usepackage{multicol}

\input xy
\xyoption{all} 

\usepackage[pdftex,plainpages=false,hypertexnames=false,pdfpagelabels]{hyperref}
\newcommand{\arxiv}[1]{\href{http://arxiv.org/abs/#1}{\tt arXiv:\nolinkurl{#1}}}
\newcommand{\arXiv}[1]{\href{http://arxiv.org/abs/#1}{\tt arXiv:\nolinkurl{#1}}}

\newcommand{\googlebooks}[1]{(preview at \href{http://books.google.com/books?id=#1}{google books})}

\usepackage{xcolor}
\definecolor{dark-red}{rgb}{0.7,0.25,0.25}
\definecolor{dark-blue}{rgb}{0.15,0.15,0.55}
\definecolor{medium-blue}{rgb}{0,0,.8}
\definecolor{DarkGreen}{RGB}{0,150,0}
\definecolor{rho}{named}{red}
\hypersetup{
   colorlinks, linkcolor={purple},
   citecolor={medium-blue}, urlcolor={medium-blue}
}

\usepackage{longtable}
\usepackage{fullpage}

\theoremstyle{plain}
\newtheorem{thm}{Theorem}[section]
\newtheorem*{thm*}{Theorem}
\newtheorem{thmalpha}{Theorem}

\newtheorem*{cor*}{Corollary}

\newtheorem*{conj*}{Conjecture}
\newtheorem*{lem*}{Lemma}
\newtheorem{lem}[thm]{Lemma}
\newtheorem{prop}[thm]{Proposition}

\newtheorem*{quest*}{Question}
\newtheorem*{claim*}{Claim}

\theoremstyle{definition}

\newtheorem{defn}[thm]{Definition}

\theoremstyle{remark}

\newtheorem{rem}[thm]{Remark}

\numberwithin{equation}{section}

\DeclareMathOperator{\tDiff}
{\mathrm{D}\!\widetilde{\,i\hspace{1.5mm}}\hspace{-1.5mm}\mathrm{ff}} 
\DeclareMathOperator{\Diff}{Diff}

\DeclareMathOperator{\Ann}{Ann}

\DeclareMathOperator{\tAnn}{A\widetilde{\!\!\phantom{i}n\phantom{i}\!\!}n}

\DeclareMathOperator{\Bigon}{Big}

\DeclareMathOperator{\id}{id}

\DeclareMathOperator{\Exp}{Exp}
\DeclareMathOperator{\Vir}{Vir}

\newcommand{\comment}[1]{}

\newcommand{\be}{\begin{enumerate}[label=(\arabic*)]}
\newcommand{\ee}{\end{enumerate}}

\newcommand{\Z}{\mathbb{Z}}

\newcommand{\C}{\mathbb{C}}
\newcommand{\CC}{\mathbb{C}}

\newcommand{\ip}[1]{\langle #1 \rangle}

\newcommand{\norm}[1]{\left\| #1 \right\|}
\newcommand{\abs}[1]{\left| #1 \right|}

\def\semicolon{;}
\def\applytolist#1{
    \expandafter\def\csname multi#1\endcsname##1{
        \def\multiack{##1}\ifx\multiack\semicolon
            \def\next{\relax}
        \else
            \csname #1\endcsname{##1}
            \def\next{\csname multi#1\endcsname}
        \fi
        \next}
    \csname multi#1\endcsname}

\def\calc#1{\expandafter\def\csname c#1\endcsname{{\mathcal #1}}}
\applytolist{calc}QWERTYUIOPLKJHGFDSAZXCVBNM;
\def\bbc#1{\expandafter\def\csname bb#1\endcsname{{\mathbb #1}}}
\applytolist{bbc}QWERTYUIOPLKJHGFDSAZXCVBNM;
\def\bfc#1{\expandafter\def\csname bf#1\endcsname{{\mathbf #1}}}
\applytolist{bfc}QWERTYUIOPLKJHGFDSAZXCVBNM;
\def\sfc#1{\expandafter\def\csname s#1\endcsname{{\sf #1}}}
\applytolist{sfc}QWERTYUIOPLKJHGFDSAZXCVBNM;

\newcommand{\noshow}[1]{}
\newcommand{\MR}[1]{}

\def\hookplus{\!\tikz{
\useasboundingbox (-.03,0) rectangle (.01,.1); \node[scale=.85] at (.32,.25) {$\scriptscriptstyle +$};
}
\hookrightarrow}

\usetikzlibrary{shapes}
\usetikzlibrary{backgrounds}
\usetikzlibrary{decorations,decorations.pathreplacing,decorations.markings}
\usetikzlibrary{fit,calc,through}
\usetikzlibrary{external}
\tikzset{
	super thick/.style={line width=3pt}
}
\tikzstyle{shaded}=[fill=red!10!blue!20!gray!30!white]
\tikzstyle{unshaded}=[fill=white]
\tikzstyle{empty box}=[circle, draw, thick, fill=white, opaque, inner sep=2mm]
\tikzstyle{annular}=[scale=.7, inner sep=1mm, baseline]
\tikzstyle{rectangular}=[scale=.75, inner sep=1mm, baseline=-.1cm]
\tikzstyle{mid>}=[decoration={markings, mark=at position 0.5 with {\arrow{>}}}, postaction={decorate}]
\tikzstyle{mid<}=[decoration={markings, mark=at position 0.5 with {\arrow{<}}}, postaction={decorate}]
\tikzstyle{over}=[double, draw=white, super thick, double=]

\usepackage{ textcomp }
\def\pounds#1{} 

\title{Integrating positive energy representations of the Virasoro algebra}
\author{Andr\'e G. Henriques, James E. Tener}
\date{}

\begin{document}

\maketitle

\abstract{
We show that every unitary positive energy representation $W$ of the Virasoro algebra exponentiates to a holomorphic $*$-representation of the semigroup of annuli by bounded operators on the Hilbert space completion of $W$.
We use this to show that every representation of the Virasoro conformal net also carries a  representation of the semigroup of annuli of the same kind. 
}

\setcounter{tocdepth}{1}
\tableofcontents

\section{Introduction}

Fix $c\in \bbR_{\ge0}$.
The Virasoro Fr\'echet Lie algebra is the Lie algebra
\[
\textstyle \bigg\{\sum_{n\in \bbZ} a_n L_n+ k\text{\bf 1} \,\bigg|\, 
\begin{aligned}
&a_n, k \in \bbC\\[-1mm]
&\text{$a_n$ is rapidly decreasing as $|n|\to\infty$}
\end{aligned}
\bigg\}
\]
with $\text{\bf 1}$ central, and Lie bracket given by
$[L_m, L_n] = (m-n)L_{n+m} + \frac{c}{12} (m^3-m) \delta_{n+m,0}\cdot \text{\bf 1}$.
Rewriting $L_n = z^{n+1}\tfrac{\partial}{\partial z}$, and omitting the symbol $\text{\bf 1}$, an alternative formula for the Virasoro Lie bracket is provided by
\[
\big[f(z)\tfrac{\partial}{\partial z},g(z)\tfrac{\partial}{\partial z}\big]
\;=\;
(f'g-gf')
\tfrac{\partial}{\partial z}
+
\omega_{Vir}\big(f(z)\tfrac{\partial}{\partial z},g(z)\tfrac{\partial}{\partial z}\big),
\]
where
\begin{equation}\label{def: omega_Vir}
\omega_{Vir}\big(f(z)\tfrac{\partial}{\partial z},g(z)\tfrac{\partial}{\partial z}\big)
\;=\; \tfrac c{12}\int_{S^1}\tfrac{\partial^3 f}{\partial z^3}(z)\,g(z)\,\tfrac{dz}{2\pi i}
\end{equation}
is the so-called \emph{Virasoro cocycle}.

In our companion paper \cite{HenriquesTener24ax}, we showed that the positive cone $\{f(z)\tfrac{\partial}{\partial z}+k\text{\bf 1}:\mathrm{Re}(f(z)/z)\le 0\}$ inside the Virasoro Fr\'echet Lie algebra 
exponentiates to a central extension $\tAnn_c$ of the the Segal-Neretin  semigroup of annuli
\[
\hspace{-2mm}\Ann:=
\left\{
\parbox{13cm}{equiv.\,classes of pairs of smooth Jordan curves $\gamma_{in},\gamma_{out}:S^1\to \bbC$ satisfying
\centerline{$\mathrm{Int}(\gamma_{in})\subseteq \mathrm{Int}(\gamma_{out})$, where $(\gamma_{in},\gamma_{out})\sim(\gamma'_{in},\gamma'_{out})$ if\, $\exists$ a diffeomorphism}
\centerline{$f:\bbC\to \bbC$ \,s.t.\,\,$f\circ \gamma_{in/out}=\gamma'_{in/out}$, and $f|_{\mathrm{Int}(\gamma_{out}) \setminus \overline{\mathrm{Int}(\gamma_{in})}}$ is holomorphic.}}
\right\}
\]
Here, $\mathrm{Int}(\gamma)$ denotes the open disc whose boundary is the image of the Jordan curve~$\gamma$.\medskip

The main results of this paper are:

\begin{thmalpha}\label{thm: intro thm 1}
    Every positive energy representation of $W$ of the Virasoro algebra $\Vir_c$ exponentiates to a holomorphic $*$-representation by bounded operators of $\tAnn_c$ on the Hilbert space completion of $W$.
\end{thmalpha}

\begin{thmalpha}\label{thm: intro thm 2}
    Every representation $\cH$ of the Virasoro conformal net induces a holomorphic $*$-representation of $\tAnn_c$.
\end{thmalpha}

Theorem~\ref{thm: intro thm 1} is proven in the body of the aritlce as Theorem~\ref{thm: representation exists} (existence) and Theorem~\ref{thm: representation is holomorphic} (holomorphicity).
The formula defining the representation is given in \eqref{eq: Definition of the action}.
A weaker version of Theorem~\ref{thm: intro thm 1} was first announced by Neretin in~\cite{Neretin90}.
Theorem~\ref{thm: intro thm 2} is proven in Section~\ref{sec:7}, and its main content is to extend the result of Theorem~\ref{thm: intro thm 1} from direct sums of irreducible representations of $\Vir_c$ to direct integrals.

\subsection*{Acknowledgements}

The second author was supported by ARC Discovery Project DP200100067.
For the purpose of Open Access, the authors have applied a CC BY public
copyright licence to any Author Accepted Manuscript (AAM) version arising from this
submission.

\section{The semigroup of annuli} \label{sec: semigroup of annuli}

The semigroup of annuli was introduced by G.~Segal and Y.~Neretin \cite{SegalDef,Neretin90}.
One of their great insights is that this semigroup,
even though it is not a group, is a kind of complexification of the group $\Diff(S^1)$ of orientation preserving diffeomorphisms of~$S^1$.
In our recent paper \cite{HenriquesTener24ax} we introduced a variation of this semigroup, which allows the annuli to have their two boundary circles intersect -- this is the semigroup we shall be further studying in the present paper.

Let us use the notation $\gamma:S^1\hookplus \bbC$
to indicate that $\gamma$ is a smooth embedding with winding number $+1$.
For such an embedding $\gamma$,  
we denote by $\mathrm{Int}(\gamma)\subset \CC$ the open disc bounded by the image of $\gamma$.
An \emph{embedded annulus}
is a subset of $\CC$ of the form
\[
A\,=\,\overline{\mathrm{Int}(\gamma_{out})}\setminus \mathrm{Int}(\gamma_{in})\subset \CC,
\]
where $\gamma_{in}, \gamma_{out}:S^1\hookplus \CC$ are embeddings satisfying $\mathrm{Int}(\gamma_{in}) \subseteq \mathrm{Int}(\gamma_{out})$.
Let us write $\varphi_{in/out}:S^1\to A$ for the maps $\gamma_{in/out}$ when thought of as maps $S^1\to A$.

\noindent
Note that we allow the images of $\gamma_{in}$ and $\gamma_{out}$ to have non-trivial intersection.
Here is\\[-4mm]
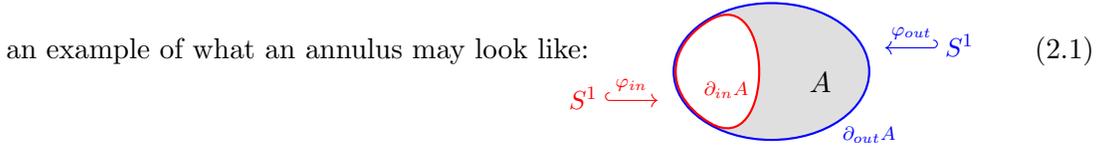
\begin{equation}\label{eq: pic of bigon}
\hspace{-.7cm}\text{an example of what an annulus may look like:} \hspace{-.4cm}
\tikz[baseline=5, yscale=.7, xscale=-1]{
\draw[thick, blue, fill=gray!25] (0,0) circle (1.3);
\draw[thick, red, fill=white] (-50:1.273) arc (-50:50:1.273) to[in=-140, out=140, looseness=1.5] (-50:1.273);
\node at (-.65,-.2) {$A$};
\node[red, scale=.9] at (.6,-.35) {$\scriptstyle \partial_{in}A$};
\node[blue, scale=.9] at (-1.3,-1.2) {$\scriptstyle \partial_{out}A$};
\node[blue, scale=.9] (Aleft) at (-2.5,.5) {$S^1$}; \draw[blue, left hook->] (Aleft.east)+(.6,-.05) --node[blue, above, scale=.9]{$\scriptstyle \varphi_{out}$} +(1.3,-.05);
\node[red, scale=.9] (Aright) at (2.5,-.5) {$S^1$}; \draw[red, right hook->] (Aright.west)+(-.6,-.05) --node[red, above, scale=.9]{$\scriptstyle \varphi_{in}$} +(-1.3,-.05);
}
\end{equation}

\noindent
We equip $A$ with the sheaf $\cO_A$ of functions that are continuous on $A$, holomorphic on $\mathring A$, and smooth on $\partial_{in}A$ and on $\partial_{out}A$. Namely, for a (relatively) open subset $U\subset A$, we set
\begin{equation}\label{eq: def of O_A}
\cO_A(U) := \left\{
\parbox{9cm}
{continuous functions $U\to \bbC$ that are holomorphic on $\mathring A\cap U$, smooth on $\partial_{in}A\cap U$, and smooth on $\partial_{out}A\cap U$}
\right\}
\end{equation}

\noindent
An \emph{abstract annulus} is a triple $(A, \varphi_{in},\varphi_{out})$ where $A$ is a locally ringed space isomorphic to an embedded annulus, and $\varphi_{in/out}:S^1\hookrightarrow A$ are the  two boundary inclusions.
An annulus $A$ is called \emph{thick} if $\partial_{in}A\cap\partial_{out}A=\emptyset$, in which case it is diffeomorphic to $S^1\times [0,1]$.

\begin{defn}
The \emph{semigroup of annuli} $\Ann$ is the set of isomorphism classes of abstract annuli.
The composition operation
\[
\cup:\Ann\times \Ann\to \Ann
\]
is defined as a pushout in the category of (locally) ringed spaces, and is well defined by \cite[Prop.~3.5]{HenriquesTener24ax}.
We equip $\Ann$ equip with the quotient topology of the space of pairs
\(
\{\gamma_{in/out}:S^1\hookplus \bbC \,|\, \mathrm{Int}(\gamma_{in}) \subseteq \mathrm{Int}(\gamma_{out})\},
\)
where the latter is topologised as a subset of $C^\infty(S^1,\bbC)^{\oplus 2}$.
\end{defn}

The semigroup of annuli is equipped with an involution
\begin{equation}
\label{eq: dag}
\dagger: \Ann \to \Ann
\end{equation}
which sends an annulus $A$ to the same annulus equipped with the opposite complex structure: $f\in \cO_A(U)\Leftrightarrow \bar f\in \cO_{A^\dagger}(U)$. The incoming and outgoing circles are exchanged, but their parametrizations remain the same.
The group $\Diff(S^1)$ admits an obvious embedding
\begin{equation}\label{eq: annulus from diffeo}
\Diff(S^1)\hookrightarrow \Ann
\end{equation}
which sends a diffeomorphism $\psi$ to the completely thin annulus $A_\psi:=(A{=}S^1,\varphi_{in}{=}\psi,$ $\varphi_{out}{=}\id)$.
This inclusion intertwines the involution $g\mapsto g^{-1}$ on $\Diff(S^1)$ with the involution $\dagger$ on $\Ann$.

We equip $\Ann$ with the complex diffeology which declares a map $M\to \Ann$ from a finite dimensional complex manifold $M$ to be holomorphic if it admits local lifts to $C^\infty(S^1,\bbC)^{\oplus 2}$ which are holomorphic.
By \cite[Lem~2.8]{HenriquesTener24ax}, this is equivalent to the existence of a global lift $M\to C^\infty(S^1,\bbC)^{\oplus 2}$ which is holomorphic.
A holomorphic map $M\to \Ann$ is also called a \emph{holomorphic family of annuli} parametrised by $M$ (see \cite[Def~2.6]{HenriquesTener24ax} for an equivalent description of holomorphic families of annuli).

In \cite[Cor~3.9]{HenriquesTener24ax}, we established that $\Ann$ is a semigroup in the category of complex diffeological spaces, meaning that if $M$ is a finite dimensional complex manifold and $f_1,f_2:M\to \Ann$ are holomorphic maps, then $\cup\circ(f_1\times f_2):M\to \Ann$ is also holomorphic.
All the operations performed in \cite[\S3.2]{HenriquesTener24ax} are continuous with respect to the $C^\infty$ topology\footnote{Those operations are: (1) solving the Riemann mapping problem for simply connected domains with smooth boundary, and (2) taking the inverse of an operator of the form identity plus smoothing operator.}, so $\Ann$ is also a semigroup in the category of topological spaces.

\begin{defn}
A \emph{framing} of an annulus $A\in \Ann$ is a smooth surjective map \begin{align*}
h\,:\,S^1\times[0,t_0]\longrightarrow\,\, &A
\\
(\theta,t)\,\,\,\mapsto\,\, h&(\theta,t)
\end{align*}
such that $h|_{S^1\times\{t\}}$ is an embedding for all $t\in [0,t_0]$
(in particular $\tfrac{\partial h}{\partial\theta} \neq 0$), whose Jacobian determinant is non-negative, and that is compatible with the boundary parametrizations in the sense that $h|_{S^1\times\{0\}}=\varphi_{in}$ and $h|_{S^1\times\{t_0\}}=\varphi_{out}$.

If $f=(f_{in},f_{out}):M\to C^\infty(S^1,\bbC)^{\oplus 2}$ represents a holomorphic family of annuli, then 
a \emph{holomorphic family of framings} is a map
$\underline{h}:M\times S^1\times [0,t_0]\to \CC$
whose associated map
$M\to C^\infty(S^1\times [0,t_0], \CC)$
is holomorphic, and
that satisfies
$\underline{h}|_{M\times S^1\times\{0\}}=f_{in}$,
$\underline{h}|_{M\times S^1\times\{t_0\}}=f_{out}$,
and each 
$\underline{h}|_{\{m\}\times S^1\times [0,t_0]}$ is a framing.

A framing has \emph{sitting instants} if there exists $\varepsilon>0$ such that for all $\theta\in S^1$ the functions
$h|_{\{\theta\}\times [0,\varepsilon]}$ and 
$h|_{\{\theta\}\times [t_0-\varepsilon,t_0]}$ are constant.
\end{defn}

Note that if $h:S^1\times[0,t_0]\to A$ is a framing,
the condition that the Jacobian determinant of $h$ is non-negative implies that for each subinterval $[t', t'']\subset[0,t_0]$, the subset $h(S^1\times[t',t''])\subset A$ is again an annulus, and the map
$(\theta,t)\mapsto h(\theta,t-t')$ is a framing of that annulus.

The semigroup of annuli admits a central extension \cite[\S5]{HenriquesTener24ax}
\[
\bbC \times \bbZ \to \tAnn \to \Ann
\]
which we conjecture is universal.
The central $\bbZ$ is generated by a $2\pi$-rotation, and corresponds to the fact that $\Ann$ is homotopy equivalent to $S^1$.
Fix $c\in \bbC$.
Taking the pushout along the map $\bbC \times \bbZ\to \bbC^\times \times \bbZ:(z,n)\mapsto (e^{cz},n)$,
we get a new central extension (no longer universal)
\begin{equation}\label{eq: The central extension}
\bbC^\times \times \bbZ \to \tAnn_c \to \Ann
\end{equation}
which we now describe.

\begin{defn}\label{defn: tAnnc}
The semigroup $\tAnn_c$ \pounds{50}
is the set of equivalence classes of triples $(A,h,z)$, where $A$ is an annulus, $h:S^1 \times [0,t_0] \to A$ is a framing of $A$, and $z \in \bbC^\times$, with respect to the equivalence relation that allows from rescalings of $[0,t_0]$,
and
declares $(A,h,z) \sim (A',h',z')$ if $A = A'$ in $\Ann$ and for every smooth one-parameter family of framings $\{h(\theta,t;u)\}_{u\in [0,1]}$ interpolating between $h$ and $h'$ we have 
\begin{equation}\label{eq: construction of central extension of Ann(S)}
z'\;\!z^{-1}
\,=\,\exp
\bigg[\int\!\!\!\!\int\! \omega\Big(h_t/h_\theta,h_u/h_\theta\Big) \,dtdu\bigg]
\end{equation}
where $\omega=\omega_{Vir}$.
Here, the ratio ${h_t}/{h_\theta}$
(respectively ${h_u}/{h_\theta}$)
denotes the preimage of $h_t:=\partial h/\partial t\in TA$ (respectively $\partial h/\partial u$) under the $\C$-linear extension $T_\C S^1\to TA$ of $\partial h/\partial \theta:T S^1\to TA$.

The semigroup operation is given by concatenation of framings
\[
(A,h,z)(A',h',z'):=(A\cup A',h\cup h',zz'),
\]
where $h\cup h':S^1\times [0,t_0+t_0']\to A\cup A'$ sends $(\theta,t)$ to $h(\theta,t)$ for $t\in [0,t_0']$, and $h'(\theta,t-t_0')$ for $t\in [t_0',t_0+t_0']$.
Note that the framing $h\cup h'$ of $A\cup A'$ is not always smooth, but that it is always possible to replace $h$ and $h'$ by homotopic framings for which the concatenation is smooth; this yields a well-defined semigroup operation at the level of equivalence classes.

Finally, the dagger operation sends an element $(A,h,z)\in \tAnn_c$ to $(A^\dagger,h^\dagger,\bar z)$, where $h^\dagger(\theta,t):=h(\theta,1-t)$.
\end{defn}

An element $[(A,h,z)]\in \tAnn_c$ is also written
\[
z\cdot\prod_{1\ge t\ge 0} \mathrm{Exp}(X(t)) dt,
\]
where $X(t) = (-h_t/h_\theta)|_{S^1\times \{t\}}$.
The dagger operation then reads
\begin{equation}\label{eqn: dagger of exponential of path}
\Bigg(z\cdot\prod_{1\ge t\ge 0} \mathrm{Exp}(X(t)) dt\Bigg)^\dagger
=
\bar z\cdot\prod_{1\ge t\ge 0} \mathrm{Exp}(-\overline{X(1-t)}) dt.
\end{equation}

\begin{rem}
The central extension $\tAnn_c$ described above is equivalent to the one which appears in \cite[Proposition 5.4]{HenriquesTener24ax}.
Following \cite[\S5]{HenriquesTener24ax}, let $\Ann^{\le A} = \{A_1 \in \Ann \,|\,  \exists A_2 : A_1 \cup A_2 = A\}$.
The space
$\tAnn_c$ is defined in loc.~cit.~as equivalence classes of triples $(A,\gamma,z)$ with $(A,\gamma,z)\sim (A,\gamma',z\cdot\exp(-\int_k\underline{\omega})\big)$, where $\gamma:[0,1]\to \Ann^{\le A}$ is a path from $1$ to $A$,
$k$ is a homotopy from $\gamma$ to $\gamma'$, and $\underline{\omega}$ is the left invariant $2$-form associated to the Virasosro cocycle (\cite[Definition~5.3]{HenriquesTener24ax}).

A framing $h:S^1\times[0,1]\to A$ produces a path $\gamma:[0,1]\to \Ann^{\le A}$ by 
sending $t$ to the annulus which is the image of $S^1\times[1-t,1]$ under $h$ (with the boundary parametrizations induced by $h$).
Similarly, a one-parameter family of framings $h(\theta,t;u)$ produces a homotopy 
$k:[0,1]^2\to \Ann^{\le A}$
between such paths.
All that remains to do is to identify $-\int_k\underline{\omega}=-\int\!\!\!\int k^*\underline{\omega}$ with $\int\!\!\!\int\! \omega\big(h_t/h_\theta,h_u/h_\theta\big) \,dtdu$.
We need to show that
\begin{equation}\label{eq: two omega formulas}
-k^*\underline{\omega} =
\omega\big(h_t/h_\theta,h_u/h_\theta\big) \,dtdu.
\end{equation}
Indeed, the value of $-k^*\underline{\omega}$ at a point $(t,u)$ is given by
$\omega(\ell^{-1}(-\tfrac{\partial k}{\partial t}(t,u),\tfrac{\partial k}{\partial u}(t,u)))dtdu$, where $\ell:\cX(S^1)\to T_{k(t,u)}\Ann^{\le A}$ is the isomorphism given by left translation. This
agrees with the right hand side of \eqref{eq: two omega formulas} (and the minus sign comes from the $1-t$ in the formula relating framings of an annulus $A$, and paths in $\Ann^{\le A}$).
\end{rem}

\pounds{51}

The semigroup $\tAnn_c$ admits a $C^\infty$ topology as well as a complex diffeology, making \eqref{eq: The central extension} into both a central extension of topological semigroups and a central extension of complex diffeological semigroups.
The dagger operation is antiholomorphic.
The local triviality of the projection map $\tAnn_c\to \Ann$ with respect to the $C^\infty$ topology is the content of 
\cite[Proposition 5.5]{HenriquesTener24ax},
and the local triviality of that same projection map in the sense of complex diffeologies is the content of 
\cite[Theorem 5.13]{HenriquesTener24ax}.

\section{Time-ordered exponentials of operators} \label{sec: time ordered exponentials}
In this section we outline some aspects of the general theory of evolution systems, following \cite[Ch. 5]{Pazy}.
These tools will be used in Section~\ref{sec: estimates for Virasoro} to construct representations of $\tAnn$.

We first review the special case of semigroups of operators (see \cite[Ch. 1]{Pazy}).
Let $\bbX$ be a Banach space, and let $T:[0,\infty) \to \cB(\bbX)$ be a function taking values in bounded operators on $\bbX$.
We say that $T$ is a semigroup of operators if $T(0) = 1$ and $T(s+t) = T(s)T(t)$ for all $s,t \ge 0$, and we say that $T$ is a $C_0$-semigroup if additionally it is strongly continuous, i.e. $\lim_{t \searrow 0} T(t)x = x$ for all $x \in \bbX$.
The generator of a $C_0$-semigroup $T$ is the densely defined operator $A$ given by 
\[
Ax = \lim_{t \searrow 0} \frac{A(t)x-x}{t}
\]
with domain $D(A)$ consisting of all vectors $x \in \bbX$ such that the defining limit exists (in the norm topology of $X$).
The generator of a $C_0$-semigroup is necessarily closed (in the sense of having a closed graph) and densely defined.
The following is a consequence of the Lumer-Phillips theorem \cite{LumerPhillips} that we will need in Section~\ref{sec: estimates for Virasoro}.
The results are standard, but a proof is included for the convenience of the reader.

\begin{lem}\label{lem: lumer phillips}
Let $\cH$ be a Hilbert space, and let $A$ be a closed densely defined operator on $\cH$ with domain $D(A)$. 
Suppose there exists a constant $\omega \ge 0$ such that 
\begin{align}
\operatorname{Re} \ip{Ax,x} \le \omega \norm{x}^2 &\qquad\,\,\, \forall\, x \in D(A), \nonumber \\
\operatorname{Re} \ip{A^*x,x} \le \omega \norm{x}^2 &\qquad\,\,\, \forall\,x \in D(A^*). \label{eqn: dissipative}
\end{align}
Then $A$ is the generator of a $C_0$-semigroup $T:[0,\infty) \to \cB(\cH)$ of bounded operators satisfying $\norm{T(t)} \le e^{\omega t}$, and similarly for $A^*$.
Moreover, the resolvent sets $\rho(A)$ and $\rho(A^*)$ contain the open interval $(\omega,\infty)$.
\end{lem}
\begin{proof}
Assumption \eqref{eqn: dissipative} says that $A-\omega$ and $A^* - \omega$ are dissipative operators in the sense of \cite[Def. 1.4.1]{Pazy}\footnote{Since $\cH$ is a Hilbert space, the duality set $F(x)$ from loc. cit. is given by $F(x)=\{x\}$ by the Cauchy-Schwarz inequality.}.
By \cite[Cor. 1.4.4]{Pazy} $A-\omega$ is the generator of a $C_0$-semigroup $\tilde T$ of contractions (note that the symbol $A^*$ in the cited Corollary refers to the Banach space adjoint operator, but the proof goes through the same with the Hilbert space adjoint operator used here).
Hence $A$ is the generator of the semigroup $T(t) = e^{\omega t} \tilde T(t)$ which satisfies $\norm{T(t)} \le e^{\omega t}$ as required.

Finally, to address the `moreover' suppose that $\omega'>\omega$.
By the Lumer-Phillips theorem \cite[Thm. 1.4.3(a)]{Pazy}, the image of $D(A)=D(A-\omega')$ under $A-\omega'$ is all of $\cH$, and similarly for $A^*-\omega'$.
Taking adjoints, it follows that $A-\omega'$ (and $A^*-\omega'$) are injective, and thus these operators map their domains bijectively onto $\cH$.
The inverse operators $(A-\omega')^{-1}$ and $(A^*-\omega')^{-1}$ are thus closed and defined on all of $\cH$, and therefore bounded.
We conclude that $\omega' \in \rho(A)$ and similarly for $A^*$.
\end{proof}

The story of operator semigroups is tied to the ODE $T'(t) = AT(t)$, whereas our treatment of the semigroup of annuli relies on the corresponding time-dependent problem $T'(t) = A(t)T(t)$.
In this case, the natural algebraic objects to replace semigroups are \emph{evolution systems} (see \cite[\S5.1]{Pazy}):
\begin{defn}
Let $\bbX$ be a Banach space, and let $U(s,t)$ be a two-parameter family of bounded operators on $\bbX$ indexed by parameters $0 \le s \le t \le T$. Then $U(t,s)$ is called an evolution system (on $\bbX$) if
\begin{enumerate}
\item $U(s,s)=1$ for $0 \le s \le T$
\item $U(t,r)U(r,s) = U(t,s)$ for $0 \le s \le r \le t \le T$
\item For every $x \in \bbX$, the map $(t,s) \mapsto U(t,s)x$ is continuous for $0 \le s \le t \le T$
\end{enumerate}
\end{defn}
If $T(t)$ is a semigroup, then $U(t,s):=T(t-s)$ is an evolution system, but most evolution systems do not arise in this way.
Informally, an evolution system is generated by a family $A(t)$ of unbounded operators if $U(t,s)x$ solves the following initial value problem for $t \ge s$:
\[
u'(t)=A(t)u(t), \qquad u(s) = x.
\]

We rely on the following generation theorem for evolution systems.
\begin{thm}[{\cite[Thm. 5.4.6]{Pazy}}] \label{thm: Pazy}
    Let $\bbX$ be a Banach space, and for $t \in [0,T]$ let $A(t)$ be the generator of a $C_0$-semigroup of operator $T_t(s)$ on $\bbX$.
    Let $\bbY$ be another Banach space that is continuously and densely embedded in $\bbX$, and suppose that the following hold:%
    \footnote{For convenience of comparison, we adopt the same notation $(H_1)$, $(H_2^+)$, $(H_3)$ as \cite{Pazy}. The letter $H$ stands for `hyperbolic', and $(H_2^+)$ is a strengthening of an alternate condition $(H_2)$ considered in loc.~cit. Our version of $(H_1)$ implies the one in loc.~cit.~by \cite[Thm. 5.2.2]{Pazy}.
    Our condition $(E_2)$ is apparently stronger than the one stated in the reference, but follows from the other properties by \cite[Thm. 5.4.3]{Pazy}.}
    \begin{enumerate}
    \item[$(H_1)$] There is a constant $\omega \ge 0$ such that $\norm{T_t(s)} \le e^{\omega s}$ for all $t,s \ge 0$ and the resolvent sets $\rho(A(t))$ contain the interval $(\omega,\infty)$ for all $t$.
    \item[$(H^+_2)$] There is a family of Banach space isomorphisms $Q(t):\bbY \to \bbX$ such that for every $v \in \bbY$ the function $t \mapsto Q(t)v$ is continuously differentiable $[0,T] \to \bbX$ and for all $0 \le t \le T$
    \[
    Q(t)A(t)Q(t)^{-1} = A(t) + B(t)
    \]
    where $B(t)$ is a strongly continuous family of bounded operators on $\bbX$.
    \item[$(H_3)$] For all $0 \le t \le T$ we have $\bbY \subset D(A(t))$, and the restriction $A(t)|_\bbY$ is a bounded operator $\bbY \to \bbX$. Moreover the map $t \mapsto A(t)|_\bbY$ is continuous $[0,T] \to \cB(\bbY,\bbX)$, where $\cB(\bbY,\bbX)$ is given the norm topology.
    \end{enumerate}
    Then there exists an evolution system $U(t,s)$ on $\bbX$ satisfying the following properties.
    \begin{enumerate}
    \item[$(E_1)$] $\norm{U(t,s)} \le e^{\omega(t-s)}$ for $0 \le s \le t \le T$
    \item[$(E_2)$] $\tfrac{\partial}{\partial t}U(t,s)v = A(t)U(t,s)v$ for $v \in \bbY$ and $0 \le s \le t \le T$
    \item[$(E_3)$] $\tfrac{\partial}{\partial s} U(t,s)v = - U(t,s)A(s)v$ for $v \in \bbY$ and $0 \le s \le t \le T$
    \item[$(E_4)$] $U(t,s)\bbY \subset \bbY$ for $0 \le s \le t \le T$
    \item[$(E_5)$] For any fixed $v \in \bbY$, the map $(t,s)\mapsto U(t,s)v: \{0 \le s \le t \le T\}\to \bbY$ is continuous. 
    \end{enumerate}
    The partial derivatives in $(E_2)$ and $(E_3)$ are all taken in the norm topology on $\bbX$, and are interpreted as one-sided derivatives when this is required by the situation.
    Moreover, $U(t,s)$ is the unique evolution system satisfying $(E_1)$-$(E_3)$.
\end{thm}

\begin{defn}\label{defn: time ordered exponential}
If $U(t,s)$ is the evolution system corresponding to a family $A(t)$ of generators in the sense of conditions $(E_1)$-$(E_5)$ of Theorem~\ref{thm: Pazy} then we write
\[
U(t,s) = \prod_{t \ge \tau \ge s} \Exp(A(\tau)d\tau),
\]
and we refer to this operator as the time-ordered exponential of the family $A(t)$.
\end{defn}

\begin{rem}\label{rem: explicit construction of evolution system}
    The evolution system $U(t,s)$ of Theorem~\ref{thm: Pazy} is constructed explicitly as a strong limit of a family of evolution systems $U_n(t,s)$ corresponding to piecewise constant approximations to the family $\{A(t)\}_{0 \le t \le T}$.
    These evolution systems $U_n(t,s)$ can be described explicitly.
    Fix a nonnegative integer $n$, and set $t_k = \frac{k}{n}T$ for $k=0,\ldots,n$.
    Let $T_t$ be the semigroup generated by $A(t)$, and for $0 \le s \le t \le T$ we follow \cite[Eq. (5.3.5)]{Pazy} and let
    \[
    U_n(t,s) = 
    \left\{
    \begin{array}{cl}
    T_{t_j}(t-s) & \text{for } t_j \le s \le t \le t_{j+1}\\
    T_{t_k}(t-t_k) \left[\prod_{j=\ell+1}^{k-1} T_{t_j}(\tfrac{T}{n}) \right] T_{t_\ell}(t_{\ell+1}-s) & \text{for } t_k \le t \le t_{k+1}, t_\ell \le s \le t_{\ell+1}\\ & \text{ and } k > \ell.
    \end{array}
    \right.
    \]
    By \cite[Eq. (3.15)]{Pazy} we have $U_n(t,s) \to U(t,s)$ in the strong operator topology.
\end{rem}

Consider now the scenario where we have a triple of densely and continuously embedded Banach spaces:
\[
\bbX_0 \supset \bbX_1 \supset \bbX_2,
\]
and unbounded operators $A(t)$ on $\bbX_0$.
Let $\tilde A(t)$ denote the restriction of $A(t)$ to $\bbX_1$ (also known as the part of $A(t)$ in $\bbX_1$), which is to say the unbounded operator on $\bbX_1$ with domain 
\[
D(\tilde A(t)) := \{v \in D(A(t)) \cap \bbX_1 \, : \, A(t)v \in \bbX_1\}
\]
given by the rule $\tilde A(t)v = A(t)v$.
Suppose the family $A(t)$ satisfies the conditions $(H_1)$, $(H_2^+)$, and $(H_3)$ for the pair of Banach spaces $(\bbX=\bbX_0,\bbY=\bbX_1)$, which produces an evolution system $U(t,s)$ of bounded operators on $\bbX_0$ by Theorem~\ref{thm: Pazy}.
As a consequence of \cite[Lem. 5.4.4]{Pazy}, the restrictions $\tilde A(t)$ generate $C_0$-semigroups $\tilde T_t$ of bounded operators on $\bbX_1$, and in fact $\tilde T_t = T_t|_{\bbX_1}$.
If we additionally assume that the $\tilde A(t)$ satisfy $(H_1)$, $(H_2^+)$, and $(H_3)$ for the pair $(\bbX_1,\bbX_2)$, then we obtain an evolution system $\tilde U(t,s)$ of bounded operators.

\begin{rem}\label{rem: evolution system on subspace}
By Remark~\ref{rem: explicit construction of evolution system}, the evolution systems $U(t,s)$ and $\tilde U(t,s)$ can be obtained as strong limits
\[
U(t,s) = \lim_k T_{x_1}(s_1) \cdots T_{x_k}(s_k), \qquad \tilde U(t,s) = \lim_k \tilde T_{x_1}(s_1) \cdots \tilde T_{x_k}(s_k),
\]
and thus $\tilde U(t,s) = U(t,s)|_{\bbX_1}$.
\end{rem}

Of course, we could perform the same analysis with a chain of Banach spaces
\[
\bbX_0 \supset \bbX_1 \supset \bbX_2 \supset \cdots,
\]
along with unbounded operators $A(t)$ on $\bbX_0$.
We say that the family $A(t)$ satisfies conditions $(H_1)_n$, $(H^+_2)_n$, $(H_3)_n$ if the restriction of $A(t)$ to $\bbX_n$ satisfies the corresponding condition for the pair of Banach space $(\bbX=\bbX_n,\bbY=\bbX_{n+1})$.
If these conditions hold for all $n=0,1,2,\ldots$, then Theorem~\ref{thm: Pazy} yields an evolution system on each space $\bbX_n$, which can all be obtained as the restriction of the evolution system $U(t,s)$ generated on $\bbX_0$.

\begin{lem}\label{lem: HORSE}
\pounds{49}
Let 
$
\bbX_0 \supset \bbX_1 \supset \bbX_2 \supset \bbX_3
$ 
be a chain of continuously and densely embedded Banach spaces, and let $A(p,t)$ be a family of unbounded operators on $\bbX_0$ indexed by $p,t \in [0,1]$. 
For each $p$, suppose that the family $\{A(p,t)\}_{t\in[0,1]}$ satisfies the conditions $(H_1)_n$, $(H_2^+)_n$, $(H_3)_n$ of Theorem~\ref{thm: Pazy} for $n=0,1,2$, and let $U_p$ be the resulting evolution system.
Suppose that the family $A(p,t)$ satisfies the following uniformity conditions:
\begin{enumerate}
    \item[i)] The constant $\omega$ of $(H_1)_n$ may be chosen independent of $p$ for $n=0,1,2$.
    \item[ii)] The assignment $p \mapsto A(p,t)|_{\bbX_1}$ is $C^1$ as a map $[0,1] \to L^\infty([0,1], \cB(\bbX_1,\bbX_0))$.
    \item[iii)] The assignment $p \mapsto A(p,t)|_{\bbX_2}$ is continuous as a map $[0,1] \to L^\infty([0,1], \cB(\bbX_2,\bbX_1))$
\end{enumerate}
The map $p \mapsto U_p(1,0)$ is differentiable $[0,1] \to \cB(\bbX_2,\bbX_0)$ and
\begin{equation}\label{eqn: differentiating family of evolution systems}
\partial_p \, U_p(1,0) = \int_0^1 U_p(1,x) \, \partial_p A(p,x) \, U_p(x,0) \, dx.
\end{equation}
\end{lem}
\begin{proof}
We begin by analyzing the integrand of \eqref{eqn: differentiating family of evolution systems}.
First observe that by assumption (i) and property $(E_1)$ the norms $\norm{U_p(t,s)}_{0\to0}$ and $\norm{U_p(t,s)}_{1 \to 1}$ are bounded independent of $p,t,s$, and by assumption (ii) the norms $\norm{A(p,t)}_{1 \to 0}$ and $\norm{\partial_p A(p,t)}_{1 \to 0}$ are bounded independent of $(p,t)$.
Next, for $v \in \bbX_1$ we note that the map $x \mapsto U_p(x,0)v$ is continuous $[0,1] \to \bbX_1$ by $(E_5)$.
Hence by assumption (ii) the map $x \mapsto \partial_p A(p,x) \, U_p(x,0)v$ is continuous $[0,1] \to \bbX_0$.
It follows that 
\[
x \mapsto U_p(1,x) \, \partial_p A(p,x) \, U_p(x,0)v
\]
is continuous $[0,1] \to \bbX_0$ since the norms $\norm{U_p(1,x)}_{0 \to 0}$ are uniformly bounded and $U_p(1,x)$ is strongly continuous in $x$ by the definition of evolution system.
Thus for each vector $v \in \bbX_1$, the integral
\[
\int_0^1 U_p(1,x) \, \partial_p A(p,x) \, U_p(x,0)v \, dx
\]
exists in $\bbX_0$, and the right-hand side of \eqref{eqn: differentiating family of evolution systems} makes sense as a bounded operator $\bbX_1 \to \bbX_0$.

For $0 \le a \le x \le b \le 1$ and $p,p' \in [0,1]$, we differentiate in $\bbX_0$ using $(E_2)$ and $(E_3)$ to obtain
\[
\partial_x \, U_{p}(b,x)U_{p'}(x,a)v = U_{p'}(b,x)\left[A(p',x)-A(p,x)\right]U_p(x,a)v.
\]
Note that this is not a straightforward application of the product rule, but can be derived by the usual argument using $(E_4)$ for $U_p$ and the fact that $\norm{U_{p}(b,x)}_{0\to0}$ is uniformly bounded.
Using arguments as above we see that 
\[
x \mapsto U_{p'}(b,x)\left[A(p',x)-A(p,x)\right]U_p(x,a)v
\]
is continuous $[0,1] \to \bbX$, and so by the fundamental theorem of calculus we have
\begin{equation}\label{eqn: evolution difference as integral}
U_{p'}(b,a)v - U_p(b,a)v = \int_{a}^b U_{p'}(b,x)\left[A(p',x)-A(p,x)\right]U_p(x,a)v \, dx.
\end{equation}
Using the fact that $\norm{U_p(t,s)}_{1\to1}$ and $\norm{U_{p'}(t,s)}_{0 \to 0}$ are uniformly bounded, we obtain a bound on the operator norm
\begin{equation}\label{eqn: evolution difference estimate}
\norm{U_{p'}(b,a) - U_p(b,a)}_{1 \to 0} \le C \sup_{0 \le x \le 1} \norm{A(p',x)-A(p,x)}_{1 \to 0}
\end{equation}
for a constant $C$ which does not depend on $a$ and $b$.

We now wish to show that
\[
\lim_{h \to 0} \norm{\frac{U_{p+h}(1,0) - U_p(1,0)}{h} - \int_0^1 U_p(1,x) \, \partial_p A(p,x) \, U_p(x,0) \, dx}_{2\to 0} = 0.
\]
Let $\epsilon > 0$.
The norms $\norm{U_p(x,0)}_{2 \to 2}$ are uniformly bounded by condition (i) and the norms $\norm{\partial_p A(p,x)}_{2 \to 1}$ are uniformly bounded by assumption (iii).
Hence by \eqref{eqn: evolution difference estimate}, for $h$ sufficiently small we have
\begin{align*}
\norm{\int_0^1 \big[ U_{p+h}(1,x) - U_p(1,x) \big] \partial_p A(p,x) U_p(x,0) \, dx}_{2 \to 0} < \epsilon.
\end{align*}
Hence
\begin{align*}
&\norm{\frac{U_{p+h}(1,0) - U_p(1,0)}{h} - \int_0^1 U_{p}(1,x) \, \partial_p A(p,x) \, U_p(x,0) \, dx}_{2\to 0} 
\le\\
&\quad\le \epsilon + \norm{\frac{U_{p+h}(1,0) - U_p(1,0)}{h} - \int_0^1 U_{p+h}(1,x) \, \partial_p A(p,x) \, U_p(x,0) \, dx}_{2\to 0} \\
\text{(by \eqref{eqn: evolution difference as integral})} &\quad= \epsilon + \norm{\int_0^1 U_{p+h}(1,x) \left[ \frac{A(p+h,x)-A(p,x)}{h} - \partial_p A(p,x)\right]U_p(x,0) \, dx}_{2 \to 0}\\
&\quad \le \epsilon + C' \sup_{0 \le x \le 1} \norm{\frac{A(p+h,x)-A(p,x)}{h} - \partial_p A(p,x)}_{1 \to 0},
\end{align*}
where the last step used the uniform boundedness of $\norm{U_{p+h}(1,x)}_{0\to 0}$ and $\norm{U_p(x,0)}_{0 \to 0}$.
By assumption (ii), this last expression converges to $0$ as $h \to 0$, completing the proof.
\end{proof}


Finally, we close the section by observing that time-ordered exponentials are compatible with adjoints.

\begin{lem}\pounds{15}\label{lem: adjoint of time ordered exponential}
Let $\cK \subset \cH$ be a continuous and dense embedding of Hilbert spaces, and let $\{A(t)\}_{t \in [0,1]}$ be unbounded operators on $\cH$ such that both $A(t)$ and the adjoints $B(t):=A(1-t)^*$ satisfy the assumptions of Theorem~\ref{thm: Pazy}.
Let $U(t,s)$ be the evolution system corresponding to the family $A(t)$, and let $\tilde U(t,s)$ be the evolution system corresponding to the family $B(t)$.
Then
\[
\tilde U(t,s) = U(1-s,1-t)^*.
\]
\end{lem}
\begin{proof}
Let $T_t$ be the semigroup generated by $A(t)$, and let $\tilde T_t$ be the semigroup generated by $B(t)$.
Fix $0 \le s \le t \le 1$, and for a nonnegative integer $n$ let $\Delta = \tfrac{t-s}{n}$ and 
\[
U_n(t,s) =  T_{s+(n-1)\Delta}(\Delta)\cdots T_{s+\Delta}(\Delta)T_s(\Delta).
\]
From e.g. the explicit formula for $T_t$ in \cite[Thm. 1.8.3]{Pazy} we see $T_\tau(s)^* = \tilde T_{1-\tau}(s)$.
Hence
\begin{align*}
U_n(1-s,1-t)^* &= T_{1-t}(\Delta)^* \cdots T_{1-t+(n-1)\Delta}(\Delta)^* = \tilde T_{t}(\Delta) \cdots \tilde T_{t-(n-1)\Delta}(\Delta)\\
&= \tilde T_{s+n\Delta}(\Delta) \cdots \tilde T_{s+\Delta}(\Delta).
\end{align*}
Arguing as in \cite[Thm. 5.3.1]{Pazy} (see also Remark~\ref{rem: explicit construction of evolution system}), we have
$U_n(t,s) \to U(t,s)$ in the strong operator topology.
Examining the above formula for $U_n(1-s,1-t)^*$, by the same argument we have $U_n(1-s,1-t)^* \to \tilde U(t,s)$ in the strong operator topology as well.
In particular these limits hold in the weak operator topology, for which the adjoint is continuous, and we conclude $\tilde U(t,s) = U(1-s,1-t)^*$, as required.
\end{proof}

\section{Estimates for the Virasoro algebra}\label{sec: estimates for Virasoro}

Fix $c\in \bbR_{\ge 0}$, and let
\[
\Vir_c :=
\textstyle \big\{\sum_{n\in \bbZ} a_n L_n+ k\text{\bf 1}\,\big|\, 
a_n, k \in \bbC
\big\}
\]
denote the Virasoro Lie algebra,
with Lie bracket
$[L_m, L_n] = (m-n)L_{n+m} + \frac{c}{12} (m^3-m) \delta_{n+m,0}$.
We use the symbol $\text{\bf 1}\in\Vir_c$ to denote a formal unit, and require representations of $\Vir_c$ to map $\text{\bf 1}$ to the identity operator.

A representation of $\Vir_c$ is said to have \emph{positive energy} if the associated operator $L_0$ has discrete spectrum, the spectrum is bounded from below, and all the (generalized)  eigenspaces are finite dimensional.
We will only be concerned with unitary representations here, in which $L_0$ acts as a self-adjoint operator, and by a positive energy representation we shall always mean a unitary positive energy representation.
Every such representation splits as a (possibly infinite) direct sum of irreducible ones.

By work of Goodman--Wallach \cite{GoWa85}, every 
unitary positive energy representation $W$ of $\Vir_c$
extends to a representation of the completed Lie algebra
\begin{equation}\label{eq: compl of Virc} 
\cX_c(S^1):=\Big\{\sum_{n=-\infty}^\infty a_nL_n + k\text{\bf 1} \,\Big|\, a_n, k \in \bbC, a_n\searrow 0 \text{ faster than any polynomial}\Big\} 
\end{equation}
on the Hilbert space completion $\cH$ of $W$.
Here, $\cX_c(S^1)$ is the central extension of the Lie algebra $\cX(S^1)$ of smooth vector fields $S^1$ associated to the Virasoro cocycle \eqref{def: omega_Vir}.
Given $X \in \cX_c(S^1)$, we write $\pi(X)$ for the closed unbounded operator on $\cH$ obtained as the closure of the corresponding operator with domain $W$, and by abuse of notation we write $L_n$ for the closed operator on $\cH$ corresponding to the basis of $\Vir_c$, in addition to the corresponding vector field.
The unitarity of the representation yields the identity 
\begin{equation}\label{eqn: virasoro adjoint}
\pi(f\partial_\theta)^* = -\pi(\overline{f}\partial_\theta),
\end{equation}
which is an identity of unbounded operators as a consequence of Nelson's commutator theorem (see e.g. \cite[\S3.2.2]{Weiner05}).
For $X=f\partial_\theta \in \cX_c(S^1)$, we write $X^* = -\overline{f}\partial_\theta$, so that $\pi$ is a $*$-representation.

The operators $\pi(X)$ all share a common invariant core, namely the smooth vectors for $L_0$, i.e.\,the subspace $\bigcap_{n>0} \operatorname{Dom}(L_0^n)$ of $\cH$.
Since $\pi(X+Y)$ and $\pi(X)+\pi(Y)$ agree on this common core, we have an inclusion
\begin{equation}\label{eqn: pixy inclusion 1}
\pi(X+Y) \subset \overline{\pi(X)+\pi(Y)}.
\end{equation}
This inclusion is in fact an equality of unbounded operators.

\begin{lem}\label{lem: equality unbounded operators}
Let $X,Y \in \cX_c(S^1)$. Then as unbounded operators we have
\[
\pi(X+Y) = \overline{\pi(X)+\pi(Y)} \quad\,\,\, \text{and}\quad\,\,\, \pi([X,Y]) = \overline{[\pi(X),\pi(Y)]}.
\]
\end{lem}
\begin{proof}
In light of \eqref{eqn: pixy inclusion 1} and \eqref{eqn: virasoro adjoint},
\[
\pi(X+Y)^* = \pi(X^*+Y^*) \subset \overline{\pi(X^*) + \pi(Y^*)} = \overline{\pi(X)^* + \pi(Y)^*}
\]
and taking adjoints yields
\begin{equation}\label{eq: 3rd inclusion}
\overline{\pi(X)+\pi(Y)}  \subset (\pi(X)^* + \pi(Y)^*)^* \subset \pi(X+Y).
\end{equation}
Combining \eqref{eqn: virasoro adjoint} and \eqref{eq: 3rd inclusion} yields
$
\pi(X+Y) = \overline{\pi(X)+\pi(Y)}.
$

The argument for the commutator is similar.
Since smooth vectors for $L_0$ are an invariant core for $\pi(X)$, $\pi(Y)$, and $\pi([X,Y])$, we have 
\begin{equation}\label{eqn: comm inclusion 1}
\pi([X,Y]) \subset \overline{[\pi(X),\pi(Y)]}.
\end{equation}
Using \eqref{eqn: comm inclusion 1} and \eqref{eqn: virasoro adjoint}, and the fact that the sum (resp.\,product) of adjoints of unbounded operators is contained in the adjoint of the sum (resp.\,product), we have: 
\[
\pi([Y^*,X^*]) \subset \overline{[\pi(Y)^*,\pi(X)^*]} \subset \overline{(\pi(X)\pi(Y))^* - (\pi(Y)\pi(X))^*} \subset [\pi(X),\pi(Y)]^*.
\]
Taking adjoints and invoking \eqref{eqn: virasoro adjoint} yields the reverse inclusion
\[
\overline{[\pi(X),\pi(Y)]} \subset \pi([Y^*,X^*])^* = \pi([X,Y]).
\qedhere
\]
\end{proof}

From now on, we suppress the operation of taking closure, and use the symbol $+$ to indicate taking the \emph{closure} of the sum of unbounded operators, and the bracket to indicate the \emph{closure} of the commutator of unbounded operators.

Using the canonical embedding $\cX(S^1)\to \cX_c(S^1)$ (which is not a Lie algebra homomorphism), we can also use the notation $\pi(X)$ when $X\in \cX(S^1)$.

\begin{rem}\label{rem: vector field vs smeared field}
The stress-energy tensor corresponding to the representation is the formal series $T(z)=\sum_{n=-\infty}^\infty L_n z^{-n-2}$, and if $f \in C^\infty(S^1)$ it is common to denote the corresponding smeared field by $T(f):=\sum_{n=-\infty}^\infty \hat f(n) L_n$, where $\hat f(n)$ are the Fourier coefficients of $f$ (and the closure of the operator is taken).
These notations are related by the formula $\pi(f \partial_\theta) = T(if)$.
\end{rem}

\begin{defn}
We denote by $\cX^{in}$ the cone of vector fields on $S^1$ that are inward pointing (potentially tangential to $S^1$):
\[
\cX^{in} = \{ f(\theta)\partial_\theta \in \cX(S^1) \, | \, \mathrm{Im} f(\theta) \ge 0 \}.
\]
We write $\cP^{in}$ for the space of smooth paths $X:[0,1] \to \cX^{in}$.
\end{defn}

Fix a positive energy representation of the Virasoro algebra $\Vir_c$, not necessarily irreducible, and let $\cH$ be its Hilbert space completion.
The main result of this section is to show that for $X \in \cP^{in}$ the time-ordered exponential
\[
\prod_{1\ge\tau\ge 0} \Exp(\pi(X(\tau))) d\tau 
\]
exists as a bounded operator on $\cH$.
Recall that (according to Definition~\ref{defn: time ordered exponential})  the time-ordered exponential is the evolution system with generators $\pi(X(\tau))$.
We will use Theorem~\ref{thm: Pazy} to demonstrate the existence of this evolution system, checking the conditions $(H_1)$, $(H_2^+)$, and $(H_3)$ of that Theorem.
We will in fact demonstrate the existence of the time-ordered exponential on a family of Hilbert spaces $\cH_n$, as follows.

For $n \ge 0$, let $\cH_n$ be the domain of the closed operator $(1+L_0)^n$, equipped with the norm 
\[
\norm{x}_{\cH_n} := \norm{(1+L_0)^n x}_{\cH}.
\]
For $X \in \cP_c^{in}$, we will verify the hypotheses of Theorem~\ref{thm: Pazy} for the operators $A(t)=\pi(X(t))$, regarded as unbounded operators on space $\bbX=\cH_n$, with respect to the subspace $\bbY=\cH_{n+1}$, for every $n=0,1,2,\ldots$.
We will denote by $(H_1)_n$, $(H_2^+)_n$, and $(H_3)_n$ the conditions arising in the hypothesis of the Theorem to indicate which choice of spaces are under consideration.

The beginning of the study of analytic aspects of representations of $\Vir_c$ is the energy bound of Goodman and Wallach \cite[Prop. 2.1]{GoWa85}.
The energy bound implies that $\cH_1$ is contained in the domain of $\pi(X)$ for every $X \in \cX_c(S^1)$ and moreover that $\pi(X)$ maps $\cH_{n+1}$ into $\cH_n$ as a bounded operator for every $n \in \Z_{\ge 0}$, with the following estimate
\[
\norm{\pi(X)v}_{\cH_n} \le \sqrt{2}\norm{X}_n \norm{v}_{\cH_{n+1}} + \sqrt{c/12}(\norm{X}_{n+1} \norm{v}_{\cH_{n+1\!/\!2}}+\norm{X}_{n+3/2} \norm{v}_{\cH_n})
\]
for all $v \in \cH_{n+1}$,
where the norm $\norm{X}_t$ is defined in terms of the Fourier series coefficients $\hat X(m)$ by\footnote{More precisely, if $X = f(\theta)\partial_\theta$ then $\hat X(m)$ is the $m$-th Fourier coefficient of $f$.}
\[
\norm{X}_t := \sum_{m \in \bbZ} (1+\abs{m})^t |\hat X(m)|.
\]
Following Toledano Laredo (see \cite[Proof of Thm. 6.1.1]{ToledanoLaredo99}), we therefore have
\begin{equation}
\label{eqn: energy bound} \norm{\pi(X)v}_{\cH_{n}} \le 
C\|X\|_{n+3/2}\norm{v}_{\cH_{n+1}}
\end{equation}
where $C=1+\sqrt{2}+\sqrt{c/12}$.
As an immediate consequence of \eqref{eqn: energy bound}, we also get a bound on the commutator of $\pi(X)$ with $L_0$:
\begin{equation}
\label{eqn: commutator energy bound} \norm{[L_0,\pi(X)]v}_{\cH_n} \le 
C\|X\|_{n+5/2}\norm{v}_{\cH_{n+1}}.
\end{equation}

We begin by using these estimates to give straightforward verifications of the conditions $(H_2^+)_n$ and $(H_3)_n$.

\begin{lem}[$H_2^+$] \label{lem: H2} \pounds{36}
Let $X \in \cP^{in}$, and let $Q=1+L_0$, regarded as a unitary operator $\cH_{n+1} \to \cH_n$ for any $n \in \bbZ_{\ge 0}$.
Then the operators
\[
B(t) := Q \pi(X(t)) Q^{-1} - \pi(X(t))
\]
are bounded operators on $\cH_n$, strongly continuous as a function of $t$.
\end{lem}
\begin{proof}
A priori $B(t)$ is defined on $\cH_{n+1}$, and is given by the formula
\[
B(t)v = [Q,\pi(X(t))]Q^{-1}v = [L_0,\pi(X(t))]Q^{-1}v, \qquad  v \in \cH_{n+1}.
\]
By \eqref{eqn: commutator energy bound} we have
\[
\norm{B(t)v}_{\cH_{n}} \le C \norm{X}_{n+5/2} \norm{Q^{-1}v}_{\cH_{n+1}} = C \norm{X}_{n+5/2} \norm{v}_{\cH_n},
\]
and thus $B(t)$ extends to a bounded operator on $\cH_n$.

It remains to verify that the map $t \mapsto B(t)$ is strongly continuous.
The linear map $\cX^{in} \to \cH_n$ given by $Y \mapsto [Q,\pi(Y)]Q^{-1}v$ is continuous by the above argument, and the map $[0,1] \to \cX^{in}$ given by $t \mapsto X(t)$ is continuous by definition, and so we are done.
\end{proof}

We can similarly use the energy bounds to give a short verification of $(H_3)_n$:

\begin{lem}[$H_3$]\pounds{35} \label{lem: H3}
Let $X \in \cP^{in}$ and let $n  \in \bbZ_{\ge 0}$.
Then for each $0 \le t \le 1$ the operator $\pi(X(t))$ is bounded $\cH_{n+1} \to \cH_n$, and moreover the map $t \mapsto \pi(X(t))$ is continuous $[0,1] \to \cB(\cH_{n+1},\cH_n)$.
\end{lem}
\begin{proof}
By \eqref{eqn: energy bound}, each $\pi(X(t))$ is bounded $\cH_{n+1} \to \cH_{n}$ with 
\[
\norm{\pi(X(t))}_{n+1\to n} \le |X|_{n+1}.
\]
Hence the map $\cX^{in} \to \cB(\cH_{n+1},\cH_n)$ given by $Y \mapsto \pi(Y)$ is a continuous linear map.
The map $t \mapsto X(t)$ is continuous $[0,1] \to \cX^{in}$ by definition, which establishes the claim.
\end{proof}

We now investigate the conditions $(H_1)_n$, and more generally the semigroups generated by the operators $\pi(X)$ on $\cH_{n}$.
The key technical tool will be the `quantum energy inequality' of Fewster and Hollands (\cite[Thm. 4.1]{FewsterHollands}, adapted to the circle as described in Remark 3 following the theorem, and also \cite{CarpiWeinerLocal}).
In this context, the result states that if $X=ig(\theta) \partial_\theta \in \cX$ and $g \ge 0$, then we have an inequality of operators
\begin{equation}\label{eqn: QEI}
\pi(X) \le \mu_X := \frac{c}{24}\| \partial_\theta \sqrt{g}\|_{L^2}^2,
\end{equation}
where $\partial_\theta \sqrt{g}$ is taken to be $0$ at points $\theta$ where $g(\theta)=0$, and is otherwise given by $g_\theta(\theta)/(2\sqrt{g(\theta)})$. 
More generally, if $X=f\partial_\theta \in \cX_c^{in}$, then we set
\begin{equation}\label{eqn: muX}
\mu_X:= \mu_{i \operatorname{Im} f \partial_\theta} = \frac{c}{24} \| \partial_\theta \sqrt{\operatorname{Im} f}\|_{L^2}^2.
\end{equation}
A short argument based on Taylor's theorem \cite[Eqn. A.3]{FewsterHollands} yields
\begin{equation}\label{eqn: bound of mu}
    \mu_{X} \le C \norm{f_{\theta\theta}}_{L^2}
\end{equation}
when $X=f\partial_\theta \in \cX^{in}$, for a constant $C$ that depends only on the central charge $c$.

We use the inequality \eqref{eqn: QEI} to obtain explicit estimates on the semigroups generated by the operators $\pi(X)$.

\begin{lem} \label{lem: semigroup generated by inward pointing vector field} \pounds{38}
Let $X \in \cX^{in}$, and let $\mu=\mu_{X}$ be as in \eqref{eqn: muX}.
Then $\pi(X)$ generates a $C_0$-semigroup $T(t)$ of bounded operators on $\cH$ satisfying $\norm{T(s)} \le e^{\mu s}$, and the resolvent set $\rho(\pi(X))$ contains the open interval $(\mu,\infty)$.
\end{lem}
\begin{proof}
We have $X = f \partial_\theta$ with $\operatorname{Im} f \ge 0$.
For a finite-energy vector $u \in W$
\[
\operatorname{Re} \ip {\pi(X)u,u} = \tfrac12 \ip{(\pi(X) + \pi(X)^*)u,u} = \ip{\pi(i\operatorname{Im}f\partial_\theta) u,u},
\]
with the final equality a consequence of \eqref{eqn: virasoro adjoint}.
Hence by the Fewster-Hollands bound \eqref{eqn: energy bound} (and notation \eqref{eqn: muX}) we have
\[
\operatorname{Re} \ip{\pi(X)u,u} \le \mu_X \norm{u}^2
\]
for all $u \in W$, and since $W$ is a core for $\pi(X)$ (by definition), this inequality holds for all $u \in D(\pi(X))$.
The same argument implies
\[
\operatorname{Re} \ip{\pi(X)^*u,u} \le \mu_X \norm{u}^2
\]
for all $u \in D(\pi(X)^*)$, where we note that $W$ is a core for $\pi(X)^*$ by \eqref{eqn: virasoro adjoint}.
By Lemma~\ref{lem: lumer phillips}, the operators $\pi(X)$ and $\pi(X)^*$ generate $C_0$-semigroups satisfying the indicated growth condition on the norm, and the resolvent sets contain $(\mu,\infty)$.
\end{proof}

\begin{lem}[$H_1$] \label{lem: H1} \pounds{37}
Let $X \in \cP^{in}$ and let $n \in \bbZ_{\ge 0}$.
Then for every $t \in [0,1]$, the restriction of $\pi(X(t))$ to $\cH_n$ generates a $C_0$-semigroup $T_t(s)$ of bounded operators on $\cH_n$. 
Moreover, there exists $\omega \ge 0$ such that $\norm{T_t(s)}_{n \to n} \le e^{\omega s}$ for all $t \in [0,1]$, and the resolvent set of every $\pi(X(t))|_{\cH_n}$ contains the interval $(\omega,\infty)$.
\end{lem}
\begin{proof}
We proceed by induction on $n$. First consider the base case $n=0$.
By Lemma~\ref{lem: semigroup generated by inward pointing vector field}, each $\pi(X(t))$ generates a $C_0$-semigroup on $\cH_0$ such that $\norm{T_t(s)} \le e^{\mu_{X(t)} s}$, and the resolvent sets $\rho(\pi(X(t)))$ contain $(\mu_{X(t)}, \infty)$.
By\eqref{eqn: bound of mu} we have $\max_{0 \le t \le 1} \mu_{X(t)} < \infty$, which furnishes our desired $\omega$.

We proceed to the inductive step, assuming the statement of the lemma holds for some fixed $n \in \bbZ_{\ge 0}$, and a particular constant $\omega_n$.
Let $Q=1+L_0$ regarded as a unitary operator $\cH_{n+1} \to \cH_n$, so each operator $Q^{-1}\pi(X(t))Q$ generates a $C_0$-semigroup $\tilde T_t$ on $\cH_{n+1}$ satisfying $\norm{\tilde T_t(s)} \le e^{\omega_n s}$, and the resolvent sets of these operators again contain $(\omega,\infty)$.
By \cite[Thm. 5.2.2]{Pazy}, the operators $Q^{-1}\pi(X(t))Q$ form a stable family (in the sense of \cite[Def. 5.2.1]{Pazy}) with stability constants $M=1$ and $\omega_n$.
Arguing as in Lemma~\ref{lem: H2}, the operator
\[
B(t):=Q^{-1}\pi(X(t))Q - \pi(X(t)) = Q^{-1}[\pi(X(t)), Q]
\]
is bounded $\cH_{n+1} \to \cH_{n+1}$ with
\[
\norm{B(t)} \le \norm{X}_{n+5/2}.
\]
Hence $K=\max_{0 \le t \le 1} \norm{B(t)} < \infty$, and by \cite[Thm. 5.2.3]{Pazy} the operators $\pi(X(t))|_{\cH_n}$ generate $C_0$-semigroups on $\cH_{n+1}$, and are a stable family on $\cH_{n+1}$ with stability constants $M=1$ and $\omega_{n+1}:=\omega_n + K$.
Thus by \cite[Thm. 5.2.2]{Pazy}, the semigroups generated by the operators $\pi(X(t))|_{\cH_{n+1}}$ on $\cH_{n+1}$ have norms bounded by $e^{\omega_{n+1} s}$, and the resolvent sets $\rho(\pi(X(t)))$ on $\cH_{n+1}$ contain $(\omega_{n+1}, \infty)$.
\end{proof}

We have proven the main result of the section:

\begin{thm}\label{thm:time ordered exponential of framing exists}
Let $W$ be a unitary positive energy representation of $\Vir_c$, not necessarily irreducible, let $\cH$ be the Hilbert space completion of $W$, and let $\pi$ be the corresponding representation of $\cX_c(S^1)$ as closed unbounded operators on $\cH$.
Let $X \in \cP^{in}$ be a smooth path in the cone $\cX^{in}$.
Then the time-ordered exponential
\[
\prod_{1\ge t \ge 0} \Exp(\pi(X(t)))dt
\]
exists.
The same holds with $\cH$ replaced by $\cH_n=D((1+L_0)^n)$, equipped with the Hilbert space norm $\norm{v}_{\cH_{n}} = \norm{(1+L_0)^n v}.$
\end{thm}
\begin{proof}
We have shown that the operators $\pi(X(t))|_{\cH_n}$ satisfy the hypotheses of Theorem~\ref{thm: Pazy} for the space $\bbX=\cH_n$ and $\bbY=\cH_{n+1}$.
They generate $C_0$-semigroups by Lemma~\ref{lem: semigroup generated by inward pointing vector field}.
These semigroups satisfy the condition $(H_1)$ by Lemma~\ref{lem: H1}, $(H_2^+)$ by Lemma~\ref{lem: H2}, and condition $(H_3)$ by Lemma~\ref{lem: H3}.
Thus by Theorem~\ref{thm: Pazy} there exists a unique evolution system satisfying conditions $(E_1)$-$(E_5)$ of the theorem statement, and the time-ordered exponential exists.
\end{proof}

\begin{rem}\label{rem: explicit norm estimates}
    In the proof of Lemma~\ref{lem: H1}, if $X \in \cX_{c}^{in}$ the condition $(H_1)_0$ is established for the family of operators $\{\pi(X(t))\}_{0 \le t \le 1}$ with $\omega = \max_{0 \le x \le 1} \mu_X$.
    Thus by \eqref{eqn: bound of mu} then the bound $(E_1)$ becomes
    \begin{equation}\label{eqn: norm bound in terms of sobolev}
    \norm{\prod_{t \ge \tau \ge s} \Exp(\pi(X(\tau))) d\tau}_{\cB(\cH_0)} \le e^{C(t-s) \max_{\tau}( \norm{X_{\theta\theta}(\theta,\tau)}_{L^2(S^1)})} 
    \end{equation}
    for some universal constant $C$.\footnote{More precisely, if we write $X(\tau)=f(\theta,\tau)\partial_\theta$, then the expression $X_{\theta\theta}(\theta,\tau)$ in the right hand side of \eqref{eqn: norm bound in terms of sobolev} refers to the second derivative $f_{\theta\theta}(\theta,\tau)$.}
    From the same proof we have the recurrence $\omega_{n+1} \le \omega_n + \max_{\tau} \norm{X(\tau)}_{n+5/2}$, and so for $n\ge1$
    \begin{equation}\label{eqn: norm bound in terms of sobolev on Hn}
    \norm{\prod_{t \ge \tau \ge s} \Exp(\pi(X(\tau))) d\tau}_{\cB(\cH_n)} \le e^{C(t-s) \max_\tau \norm{X(\tau)}_{n+5/2}}.
    \end{equation}
\end{rem}

\begin{rem}\label{rem: diff in families}
 In light of Remark~\ref{rem: explicit norm estimates}, we note that if $U \subset \bbR^k$ is open and $\{X_p\}_{p \in U}$ is a smooth family in $\cP^{in}$, then the operators $\pi(X_p(t))$ locally satisfy the uniformity condition (i) of Lemma~\ref{lem: HORSE}. The conditions (ii) and (iii) of that lemma follow from the energy bound \eqref{eqn: energy bound}, and so we may differentiate $\partial_p \prod_{1\ge t \ge 0} \Exp(\pi(X_p(t)))dt$ in $\cB(\cH_2,\cH_0)$ using the formula \eqref{eqn: differentiating family of evolution systems}.
\end{rem}

\section{Representation of the semigroup of annuli}

Let $(W,\pi)$ be a unitary positive energy representation of the Virasoro algebra $\Vir_c$.
The usual Hilbert space completion of $W$ is denoted $\cH$, and the completions of $W$ with the respect to the norms $v\mapsto\|(1+L_0)^n v\|$ are denoted $\cH_n$, so that we have a sequence of dense inclusions
\[
\cH=\cH_0 \supset \cH_1 \supset \cH_2  \supset \cH_3  \supset \ldots \supset W.
\]
The goal of this section is to integrate the above representation to a representation
\begin{equation}\label{eq: pi for annc}
\pi:\tAnn_c\to B(\cH)
\end{equation}
of the semigroup of annuli by bounded operators on the Hilbert space $\cH$.
We will show that this representation leaves each of the subspaces $\cH_n$ invariant, and that it is compatible with the dagger operation in the sense that $\langle \pi(A) \xi, \eta \rangle_{\cH} = \langle \xi, \pi(A^\dagger) \eta \rangle_{\cH}$.

For  $h:S^1\times[0,1]\to A$ a framing of some annulus $A$, 
and $X(t) := (-h_t/h_\theta)|_{S^1\times \{t\}}$,
recall that
\(
z\cdot\prod_{1\ge t\ge 0} \mathrm{Exp}(X(t)) dt
\)
represents a lift of $A$ to $\tAnn_c$.
We will define \eqref{eq: pi for annc} by setting
\begin{equation}\label{eq: Definition of the action}
\pi\left(z\cdot\!\!\prod_{1\ge t \ge 0} \Exp(X(t))dt\right)
\,\,:=\,\,
z\cdot\!\!\prod_{1\ge t \ge 0} \Exp\big(\pi(X(t))\big)dt,
\end{equation}
where the time-ordered exponential on the right-hand side exists by Theorem~\ref{thm:time ordered exponential of framing exists}.
The fact the prescription \eqref{eq: Definition of the action} is well-defined is the main result of this section.
We begin with two technical lemmas about framings.

\begin{lem}\label{lem: SPADE}
\pounds{55 $\spadesuit$}
Let $h:S^1 \times [0,1] \to  A$ be a framing of a partially thin annulus $A$. Then $h$ can be homotoped through framings to one which is a diffeomorphism in the bulk (i.e. a diffeomorphism from $h^{-1}(\mathring A) \to \mathring A$).
\end{lem}
\begin{proof}

For every point $z\in A$, pick a preimage $(\theta,t)$ under $h$, and let $v_z\in T_zA$ be the image of $\partial_\theta$ under the map $dh:T_{(\theta,t)}(S^1\times [0,1])\to T_zA$.
The vector $v_z$ is independent of $(\theta,t)$ up to a scalar, and the assignment $z\mapsto \bbR v_z$ defines a continuous (but typically not smooth) distribution of lines.

Pick a smooth foliation $F$ of $A$ which is transverse to the above distribution of lines. Such a foliation can be for example obtained by smoothing the orthogonal to the above distribution, and then integrating it to a foliation of $A$. Let $F_\theta\subset A$ denote the leaf of $F$ that intersects $\partial_{in}A$ at $h(\theta,0)$.

For each $t$, the map $\theta\mapsto F_\theta\cap h(S^1\times\{t\})$ is a diffeomorphism $S^1\to h(S^1\times\{0\})$, hence defines a diffeomorphism $\alpha_t:S^1\to S^1$.
Moreover, $\alpha_t$ depends smoothly on $t$. Let $u\mapsto \alpha_{t,u}:[0,1]\to\Diff(S^1)$  be the `straight line homotopy' connecting $\alpha_t$ to $\id_{S^1}$.

Then $\{h_u:S^1 \times [0,1] \to  A\}_{u\in [0,1]}$ given by $h_u(\theta,t) = h(\alpha_{t,u}(\theta),t)$ is a one parameter family of framings interpolating between our original framing $h$ and a new framing $h_1$. This new framing has the property that $h_1(\{\theta\}\times [0,1]) = F_\theta$.

The images under $h_1$ of the intervals $\{\theta\}\times[0,1]$ are given by the leaves of $F$, so
we may finish the proof by performing a straight line homotopy in the $[0,1]$ direction between our map $h_1$ and the constant speed parametrizations $\{\theta\}\times[0,1] \to F_\theta$.
This connects $h_1$ to a new framing which is a diffeomorphism in the bulk.
\end{proof}

\begin{lem}\label{lem: DIAMOND}
\pounds{56 $\diamondsuit$}
Let $A$ be a partially thin annulus. And let $h_0,h_1: S^1 \times [0,1] \to A$ be framings. 
Then $h_0$ and $h_1$ are homotopic through framings. \end{lem}

\begin{proof}
As in the proof of the previous lemma, let $F_0$ and $F_1$ be smooth foliations of $A$ which are transverse to the one dimensional distributions associated to $h_0$ and $h_1$.
By a partition of unity argument, we may furthermore arrange that $F_0$ and $F_1$ agree on a neighborhood of the thin part of $A$.
By the previous lemma, there then exist framings $h_0'$ and $h_1'$ which are diffeomorphisms in the bulk, along with one parameter families of framings connecting $h_0$ to $h_0'$, and $h_1$ to $h_1'$.
Moreover, by construction, $h_0'$ and $h_1'$ agree in a neighborhood of the thin part of $A$.

It remains to connect $h_0'$ and $h_1'$ by a one parameter families of framings.
As $h_0'$ and $h_1'$ agree in a neighborhood of the thin part of $A$, we just need to worry about the part of $A$ outside that neighborhood.
That part is contained inside the union of finitely many 2-discs, and we may finish the proof by applying the classical result that any two diffeomorhpisms of a 2-disc are isotopic. 
\end{proof}

\begin{lem}\label{lem: well defined on equivalence classes}
\pounds{53}
Let $A$ be an annulus,  let $\{h_s:S^1\times[0,1]\to A\}_{s\in [0,1]}$ be a $1$-parameter family of framings, and
let $X := -(\partial h/\partial t)  /  (\partial h/ \partial \theta)$ and $Y := -(\partial h/\partial s)  /  (\partial h/ \partial \theta)$.
We then have
\begin{equation}\label{eqn: differ by cocycle}
\prod_{1\ge t \ge 0}\mathrm{Exp}\big(\pi(X_{1,t}) dt\big)
=
\mathrm{exp}\Big(\iint_{0 \le s,t \le 1}  \omega (X_{s,t}, Y_{s,t}) dt\Big) \Big[ \prod_{1 \ge t \ge 0}\mathrm{Exp}\big(  \pi (X_{0,t}) dt\big)\Big]
\end{equation}
\end{lem}
\begin{proof}
By uniqueness of solutions for 1st order ODEs, if $f'(r) =  \lambda (r) f(r)$, then $f(r) = \mathrm{Exp}(  \int_{0 \le s \le r} \lambda (s)) f(0)$. Certainly we know this in the one-dimensional case, and the (infinite dimensional) vector-valued case follows from the one-dimensional case by pairing against linear functionals. 
Thus, it suffices to prove:
\[
\partial_s \Big(  \prod_{1 \ge t \ge 0}\mathrm{Exp}\big(  \pi (X_{s,t}) dt\big) \Big) =
\int_{0 \le t \le 1}  \omega (X_{s,t}, Y_{s,t}) dt   \cdot   \Big[ \prod_{1 \ge t \ge 0}\mathrm{Exp}(  \pi (X_{s,t}) dt)\Big]
\]
We will show this identity as operators $\cH_2 \to \cH_0$, which will then impliy that \eqref{eqn: differ by cocycle} holds when the relevant operators are applied to vectors in $\cH_2$.
As both sides are bounded operators on $\cH_0$, the lemma will follow.

We expand (as operators $\cH_2 \to \cH_0$):
\begin{align*}
&\,\partial_s \Big(  \prod_{1 \ge t \ge 0}\mathrm{Exp}\big(  \pi (X_{s,t}) dt\big) \Big)
\\
&= \int_{0 \le x \le 1}  \prod_{1 \ge t \ge x} \mathrm{Exp}(  \pi (X_{s,t}) dt)  \big[  \partial_s   \pi (X_{s,t}) \big]\big|_{t=x}  \prod_{x \ge t \ge 0} \mathrm{Exp}(  \pi (X_{s,t}) dt)\quad \scriptstyle\text{(by Lemma~\ref{lem: HORSE}/Remark~\ref{rem: diff in families}})\hspace{1.1cm}
\\
&=  \int_{0 \le x \le 1}  \prod_{1 \ge t \ge x} \mathrm{Exp}\big(  \pi (X_{s,t}) dt\big)  \big[   \pi ( \partial_s X_{s,t}) \big]\big|_{t=x}  \prod_{x \ge t \ge 0} \mathrm{Exp}\big(  \pi (X_{s,t}) dt\big)
\end{align*}\vspace{-7mm}
\begin{align*}
&=
\int_{0 \le x \le 1}  \prod_{1 \ge t \ge x} \mathrm{Exp}\big(  \pi (X_{s,t}) dt\big)  \big[   \pi ( \partial_t Y_{s,t} -[X_{s,t},Y_{s,t}]_{\scriptscriptstyle\mathrm{Witt}})   \big]\big|_{t=x}  \prod_{x \ge t \ge 0} \mathrm{Exp}\big(  \pi (X_{s,t}) dt\big)
\quad \scriptstyle\text{(by Lemma~\ref{lem: NECESSARY})}
\\
&= 
\int_{0 \le x \le 1}  \prod_{1 \ge t \ge x} \mathrm{Exp}\big(  \pi (X_{s,t}) dt\big)  \big[   \pi ( \partial_t Y_{s,t}) - [  \pi (X_{s,t}),   \pi (Y_{s,t})]
\\[-3mm]
&
\phantom{=\int_{0 \le x \le 1}  \prod_{1 \ge t \ge x} \mathrm{Exp}\big(  \pi (X_{s,t}) dt\big)  \big[   \pi ( \partial_t Y_{s,t})} +  \omega (X_{s,t},Y_{s,t})  \big]\big|_{t=x}  \prod_{x \ge t \ge 0} \mathrm{Exp}\big(  \pi (X_{s,t}) dt\big)
\\
&\stackrel{?}{=}
\int_{0 \le t \le 1}    \omega (X_{s,t},Y_{s,t}) dt  \prod_{1 \ge t \ge 0} \mathrm{Exp}\big(  \pi (X_{s,t}) dt\big)
\end{align*}
The equality indicated with a question mark holds so long as the following operator (from $\cH_2$ to $\cH_0$) is zero:
\begin{equation}\label{eqn: needs to vanish}
\int_{0 \le x \le 1} \prod_{1 \ge t \ge x} \Exp\big( \pi(X_{s,t}) \, dt\big) \big(\pi(\partial_t Y_{s,t}) - [\pi(X_{s,t}),\pi(Y_{s,t})]\big)\big|_{t=x} \prod_{x \ge t \ge 0} \Exp\big(\pi(X_{s,t}) \, dt\big).
 \end{equation}
We will show that \eqref{eqn: needs to vanish} holds by arguing that the integrand is the derivative of a quantity which vanishes at $x=0$ and at $x=1$:
\begin{align}\label{eqn: intermediate step derivative}
&\partial_x \Big( \prod_{1 \ge t \ge x} \Exp\big( \pi(X_{s,t}) \, dt \big) \pi(Y_{s,t})\big|_{t=x} \prod_{x\ge t \ge 0} \Exp\big( \pi(X_{s,t}) \, dt \big) \Big)
\\\nonumber&=
\prod_{1 \ge t \ge x} \Exp\big( \pi(X_{s,t}) \, dt\big) \big(\pi(\partial_t Y_{s,t}) - [\pi(X_{s,t}),\pi(Y_{s,t})]\big)|_{t=x} \prod_{x \ge t \ge 0} \Exp\big(\pi(X_{s,t}) \, dt\big)
\end{align}
(as $Y_{s,t}$ vanishes identically at $t=0$ and $t=1$).
More precisely, we compute the derivative on the left-hand side of \eqref{eqn: intermediate step derivative} after applying to an arbitrary vector in $v \in \cH_3$\footnote{
For $v \in \cH_3$, function $x \mapsto  \prod_{x\ge t \ge 0} \Exp\big( \pi(X_{s,t}) \, dt \big)v$ is differentiable in $\cH_2$, the function $x \mapsto \pi(Y_{s,x})$ is differentiable in $\cB(\cH_2,\cH_1)$, and for $v \in \cH_1$ the function $x \mapsto \prod_{1 \ge t \ge x} \Exp\big( \pi(X_{s,t}) \, dt\big)v$ is differentiable in $\cH_0$.
As all operators are uniformly bounded in norm, the usual product rule can be invoked.
}.
As operators $\cH_3 \to \cH_0$,
we have:
\begin{align*}
&\,\partial_x \Big( \prod_{1 \ge t \ge x} \Exp\big( \pi(X_{s,t}) \, dt\big) \pi(Y_{s,t})|_{t=x} \prod_{x \ge t \ge 0} \Exp\big(\pi(X_{s,t}) \, dt\big) \Big)\\
=&-
\prod_{1 \ge t \ge x} \Exp\big( \pi(X_{s,t}) \, dt\big) \pi(X_{s,t})|_{t=x}\pi(Y_{s,t})|_{t=x} \prod_{x \ge t \ge 0} \Exp\big(\pi(X_{s,t}) \, dt\big)\\
&+
\prod_{1 \ge t \ge x} \Exp\big( \pi(X_{s,t}) \, dt\big) \pi(Y_{s,t})|_{t=x}\pi(X_{s,t})|_{t=x} \prod_{x \ge t \ge 0} \Exp\big(\pi(X_{s,t}) \, dt\big)\\
&+
\prod_{1 \ge t \ge x} \Exp\big( \pi(X_{s,t}) \, dt\big) \pi(\partial_t Y_{s,t})|_{t=x} \prod_{x \ge t \ge 0} \Exp\big(\pi(X_{s,t}) \, dt\big)\\
=&\phantom{+}
\prod_{1 \ge t \ge x} \Exp\big( \pi(X_{s,t}) \, dt\big) \big(\pi(\partial_t Y_{s,t}) - [\pi(X_{s,t}),\pi(Y_{s,t})]\big)|_{t=x} \prod_{x \ge t \ge 0} \Exp\big(\pi(X_{s,t}) \, dt\big).
\end{align*}
Hence \eqref{eqn: needs to vanish} vanishes on $\cH_3$, and thus by continuity on all of $\cH_2$, as required.
\end{proof}

\begin{lem}\label{lem: NECESSARY}
\pounds{54}
Let $A$ be an annulus, and let $\{h_s:S^1\times[0,1]\to A\}_{s\in [0,1]}$ be a $1$-parameter family of framings.
Let $X := -(\partial h/\partial t)  /  (\partial h/ \partial \theta)$ and $Y := -(\partial h/\partial s)  /  (\partial h/ \partial \theta)$, which we rewrite as $X := -h_t / h_\theta$, $Y := -h_s / h_\theta$.
Then 
\[
\partial_t Y - \partial_s X = Y \partial_\theta X - X\partial_\theta Y,
\]
which we rewrite as  $Y_t - X_s = [X,Y]_{\mathrm{Witt}}$.
\end{lem}

%
%
%
%
%

\begin{proof}
As
$
X_s = -(h_t / h_\theta )_s = h_\theta^{-2}(h_t h_{s \theta} - h_{st} h_\theta)$,
$Y_t = -(h_s / h_\theta )_t = h_\theta^{-2}(h_s h_{t\theta} - h_{st} h_\theta )$,
we have $Y_t - X_s = 
h_\theta^{-2}(h_sh_{t\theta} - h_t h_{s\theta})$.
It follows that
\begin{align*}
[X,Y] &= YX_\theta - XY_\theta 
\\&=
(h_s / h_\theta )\cdot(h_t / h_\theta )_\theta -
(h_t / h_\theta )\cdot(h_s / h_\theta )_\theta
\\&=
(h_s / h_\theta )\cdot (h_{t\theta} h_\theta   - h_t h_{\theta\theta} )/h_\theta^{2}
-
(h_t / h_\theta)\cdot (h_{s\theta} h_\theta  - h_s h_{\theta\theta} )/h_\theta^{2} 
\\&=
h_\theta^{-2}(h_sh_{t\theta} - h_th_{s\theta})\\&=Y_t - X_s.\qedhere
\end{align*}
\end{proof}

\begin{thm}\label{thm: representation exists}
Let $W$ be a unitary positive energy representation of the Virasoro algebra $\Vir_c$, and let $\cH$ be the Hilbert space completion of $W$.
Then \eqref{eq: Definition of the action} defines a $*$-representation of $\tAnn_c$ on $\cH$ by bounded operators. Moreover, each of the subspaces $\cH_n\subset \cH$ is invariant under that action.
\end{thm}
\begin{proof}
Recall from Definition~\ref{defn: tAnnc} that elements of $\tAnn_c$ are equivalence classes of triples $(A,h,z)$ where $h$ is a framing of the annulus $A$ and $z \in \bbC^\times$ is a complex number.
The formula \eqref{eq: Definition of the action} assigns an operator to such a triple, and it is the content of Lemmas~\ref{lem: DIAMOND} and \ref{lem: well defined on equivalence classes} that this assignment is independent of equivalence class representative.
It is immediate that this is a semigroup homomorphism. 
    
    We next verify that this is a $*$-representation.
    Indeed:
    \begin{align*}
    \pi\Bigg(\Bigg[\prod_{1\ge t\ge 0} \mathrm{Exp}(X(t)) dt{}\Bigg]^\dagger
\Bigg)
&=\pi\Bigg(\prod_{1\ge t\ge 0} \mathrm{Exp}(-\overline{X(1-t)}) dt \Bigg)\\
&= \prod_{1\ge t\ge 0} \mathrm{Exp}(\pi(-\overline{X(1-t)})) dt \\
&= \prod_{1\ge t\ge 0} \mathrm{Exp}(\pi(X(1-t))^* dt \\
&= \pi\Bigg(\prod_{1\ge t\ge 0} \mathrm{Exp}(X(t)) dt{}
\Bigg)^* .
    \end{align*}
The first equality is \eqref{eqn: dagger of exponential of path}, the second is the definition of the action of $\tAnn_c$, the 
    third is \eqref{eqn: virasoro adjoint}, and the fourth is Lemma~\ref{lem: adjoint of time ordered exponential} and the definition of the action of $\tAnn_c$.

    Finally, the operators
    \(
    \pi\Big(\prod_{1\ge t\ge 0} \mathrm{Exp}(X(t)) dt\Big)
    \)
    leave the subspace $\cH_n \subset \cH$ invariant by Theorem~\ref{thm:time ordered exponential of framing exists} (see also Remark~\ref{rem: evolution system on subspace}).
\end{proof}


Given an annulus $A$, a framing $h:S^1\times [0,1]\to A$, and $X:=-h_t / h_\theta$, by \cite[Lem~A.8]{HenriquesTener24ax}, holomorphic vector fields on $A$ correspond to functions $f$ on $S^1\times [0,1]$ that satisfy 
\begin{equation}\label{eqn: holo vector field}
f_t = X_\theta f - X f_\theta.
\end{equation}

\begin{defn} \pounds{43}
Let $X:[0,1] \to \cX^{in}$ be a smooth path (i.e. $X \in \cP^{in}$). An operator $T \in \cB(\cH)$ satisfies the \emph{Segal commutation relations} for $X$ if for every function $f$ on $S^1 \times [0,1]$ satisfying \eqref{eqn: holo vector field}  the following equation holds
\begin{equation}\label{eq: SCM}
T \circ \pi(f|_{S^1 \times \{0\}}) = \pi (f|_{S^1 \times \{1\}}) \circ T
+\big(\textstyle\int_0^1 \omega(X(\tau),f(\tau))d\tau\big)\cdot T
\end{equation}
as operators $\cH_1 \to \cH_0$.
\end{defn}

\begin{prop}\pounds{45}\pounds{44}
Let $W$ be a unitary positive energy representation of the Virasoro algebra, and let $\cH$ be its Hilbert space completion. If $X \in \cP^{in}$,
then $$\prod_{1\ge \tau\ge 0} \mathrm{Exp}\big(\pi(X(\tau))d\tau\big)\in \cB(\cH)$$  satisfies the Segal commutation relations for $X$.
\end{prop}
\begin{proof}
Set $U(t,s) = \prod_{t \ge \tau \ge s}  \Exp(\pi(X(\tau)))$.
Let $f:S^1\times [0,1]\to \mathbb C$ satisfy \eqref{eqn: holo vector field}, and
\[
V_{out}(t,s) := \big(\pi(f(t))+ \textstyle\int_s^t \omega(X(\tau),f(\tau))d\tau\big) U(t,s),
\quad
V_{in}(t,s) :=  
U(t,s) \pi(f(s)).
\]
For $\xi$ sufficiently regular, we will show that $V_{out}(t,s) \xi$ and $V_{in}(t,s) \xi$ satisfy a common characterizing initial value problem in $t$.
Using \eqref{eqn: holo vector field}, we first observe that, formally,
\[
\tfrac{d}{dt} \pi(f(t)) = \pi\big(\tfrac{d}{dt} f(t)\big) = \pi\big(X(t)' f(t) - X(t) f(t)'\big),
\]
where the prime denotes differentiation in the $\theta$ variable.

Still working formally, we thus have:
\begin{align}
\nonumber \tfrac{d}{dt} V_{out}(t,s)&= \tfrac{d}{dt} \big(\big[\pi(f(t))+ \textstyle\int_s^t \omega(X(\tau),f(\tau))d\tau\big]U(t,s)\big)\\
\nonumber &= \left[\pi\big(X(t)' f(t) - X(t) f(t)'\big)+\omega\big(X(t),f(t)\big)\right] U(t,s) \\\notag&\qquad\qquad\qquad\qquad+
\big(\pi(f(t)) +\textstyle\int_s^t \omega(X(\tau),f(\tau))d\tau\big)[ {\tfrac{d}{dt}} U(t,s) ]
\\
\nonumber &= \pi([X(t),f(t)]_{\Vir})U(t,s)
\\\notag&\qquad\qquad\qquad\qquad+ \big(\pi(f(t))+\textstyle\int_s^t \omega(X(\tau),f(\tau))d\tau\big) \pi(X(t)) U(t,s)\\\nonumber
 &=  \pi(X(t))\big(\pi(f(t))+\textstyle\int_s^t \omega(X(\tau),f(\tau))d\tau\big)  U(t,s)\\
&= \pi(X(t)) V_{out}(t,s)
\label{eqn: segal comm rels 1}
\intertext{and}
\nonumber
\tfrac{d}{dt} V_{in}(t,s) &= \tfrac{d}{dt} \big(
U(t,s) \pi(f(s))\big) \\\nonumber
&= \pi(X(t)) U(t,s) \pi(f(s))
\\&= \pi(X(t)) V_{in}(t,s).
\label{eqn: segal comm rels 2}
\end{align}

We now address analytic aspects of the calculations \eqref{eqn: segal comm rels 1} and \eqref{eqn: segal comm rels 2}.
Fix $\xi \in \cH_2.$
By property $(E_2)$ of Theorem~\ref{thm: Pazy} (applied with Hilbert space pair $(\cH_1,\cH_2)$), the map $t \mapsto U(t,s)\xi$ is differentiable $[s,\infty) \to \cH_1$ with derivative $\pi(X(t))U(t,s)\xi$.
Moreover, $t \mapsto f(t)$ is smooth (with values in $\cX(S^1)$) and $f \mapsto \pi(f)$ is bounded $\cX(S^1) \to \cB(\cH_{n+1},\cH_n)$ by \eqref{eqn: energy bound}.
Hence $t \mapsto V_{out}(t,s)\xi$ is differentiable $[s,\infty) \to \cH_0$ with derivative given by the calculation \eqref{eqn: segal comm rels 1}.
Moreover by property $(E_5)$ of Theorem~\ref{thm: Pazy}, $t \mapsto U(t,s)\xi$ is continuous as a map into $\cH_2$, and it follows that the formula for $\tfrac{d}{dt}V_{out}(t,s)$ defines a continuous function $[s,\infty) \to \cH_0$, which is to say that $V_{out}(t,s)\xi \in C^1([s,\infty),\cH_0)$.
A similar argument shows that $V_{out}(t,s)\xi \in C([s,\infty),\cH_1)$.

On the other hand, for $\xi \in \cH_2$ we have $\pi(f(s))\xi \in \cH_1$, and so by $(E_2)$ and $(E_5)$ of Theorem~\ref{thm: Pazy} we have $V_{in}(t,s)\xi \in C^1([s,\infty), \cH_0) \cap C([s,\infty),\cH_1)$, and the derivative $\tfrac{d}{dt} V_{in}(t,s)\xi$ is given as in \eqref{eqn: segal comm rels 2}.
We now invoke \cite[Thm. 5.4.3]{Pazy} which says that $t \mapsto U(t,s)\pi(f(s))\xi$ is the unique function in $C([s,\infty), \cH_1) \cap C^1([s,\infty),\cH_0)$ satisfying the ODE $\frac{d}{dt} u(t) = \pi(X(t)) u(t)$ with $u(s)=\pi(f(s))\xi$.
Hence we conclude $V_{out}(t,s)\xi = V_{in}(t,s)\xi$ for all $\xi \in \cH_2$.
As both sides are bounded operators $\cH_1 \to \cH_2$, we conclude that the identity holds for all $\xi \in \cH_1$, completing the proof.
\end{proof}

The next proposition is one we only know how to prove for paths $X \in \cP^{in}$ of the form $X=-h_t/h_\theta$, for $h$ a framing of some annulus. These paths are called \emph{geometrically exponentiable}\footnote{\cite[Conj.~1.4]{HenriquesTener24ax} states that 
all paths are geometrically exponentiable.} in the terminology of \cite{HenriquesTener24ax}:

\begin{prop}
Let $W$ be an irreducible unitary positive energy representation of $\Vir_c$, and let $X \in \cP^{in}$ be geometrically exponentiable. Then  $\prod_{1\ge \tau\ge 0} \mathrm{Exp}\big(\pi(X(\tau))d\tau\big)$ is uniquely determined up to a scalar by the Segal commutation relations \eqref{eq: SCM}.
\end{prop}

\noindent
We only sketch the argument:

\begin{proof}
Let $A$ be an annulus, and let $h:S^1\times[0,1]\to A$ be a framing such that $X=-h_t/h_\theta$. Let $T$ be an operator satisfying the Segal commutation relations \eqref{eq: SCM}. Pick annuli $A_1$ and $A_2$ such that $A_1AA_2=r^{\ell_0}$. The operators $T_i:=\pi(A_i)$ satisfy \eqref{eq: SCM} for $A_i$ by the previous proposition (with respect to some framings $h_i$ of $A_i$).
So $T_1TT_2$ satisfies the Segal commutation relations for $r^{\ell_0}$.

By letting $f$ in \eqref{eq: SCM} correspond to the pullback of $\ell_0=z\partial_z$ under the given framing of $A_1AA_2=r^{\ell_0}$, we see that $T_1TT_2$ commutes with $L_0$, hence induces a map at the level of finite energy vectors. 
By considering the pullbacks of $\ell_n = z^{n+1}\partial_z$, we see that the operator $T_1TT_2$ also satisfies commutations relations with the operators $L_n$:
\[
L_n (T_1TT_2) = r^n (T_1TT_2)L_n.
\]
Since $W$ is irreducible, it follows that $T_1TT_2=r^{L_0}$ up to a constant. By rescaling $T$, we may assume without loss of generality that constant is equal to $1$.
By construction, the operators $T_1$ and $T_2$ are injective with dense image. Hence we may conclude from the identity
$T_1 \pi(A) T_2 = r^{L_0} = T_1 T T_2$,
 that $\pi(A) = T$, as claimed. 
\end{proof}

\section{Continuity and holomorphicity}

In this section, we show that the representations of
$\tAnn_c$ constructed in Theorem~\ref{thm: representation exists} are holomorphic.
(Neretin \cite{Neretin90} claimed a similar result for the subsemigroup of ‘thick annuli,’ although his paper did not include a proof of holomorphicity.)
We will expand on what  it means for the a representation $\pi:\tAnn_c \to \cB(\cH)$ to be holomorphic after some preliminary notions.

If $\bbV$ is a locally convex topological vector space and $U \subset \bbC$ is open, a function $f:U \to \bbV$ is called holomorphic if for every $z \in U$ the limit 
\[
\lim_{z \to z_0} \frac{f(z) - f(z_0)}{z - z_0}
\]
exists in the completion of $\bbV$.
More generally, if $U \subset \bbC^n$ then $f$ is called holomorphic if it is separately holomorphic as a function of each of its $n$ variables.

We will be particularly interested in holomorphic functions valued in the bounded operators $\cB(\cH)$ on a Hilbert space $\cH$, equipped with one of the following topologies: the norm topology, the strong operator topology (or SOT, generated by the seminorms $T \mapsto \norm{Tv}$), or the weak operator topology (or WOT, generated by the seminorms $T \mapsto \abs{\ip{Tv,u}}$).
As it turns out, the notion of holomorphic function into $\cB(\cH)$ is independent of the choice of topology (among the ones listed above). We include a proof of this well-known fact for the convenience of the reader:

\begin{lem}\label{lem: B(H) holomorphicity}
Let $U \subset \bbC^n$ be open and let $f:U \to \cB(\cH)$. Then $f$ is holomorphic with respect to the norm topology on $\cB(\cH)$ if and only if it is holomorphic with respect to the SOT, if and only if it is holomorphic with respect to the WOT.
\end{lem}
\begin{proof}
Without loss of generality we may consider $U \subset \bbC$.
Holomorphicity being a local property, we may further assume that $U$ is a disc, and that $f$ extends holomorphically past the boundary of $U$.
We need only show that a WOT-holomorphic function $f:U \to \cB(\cH)$ is in fact norm holomorphic.
We begin by showing that such a function is norm continuous.

Fix a point $z_0 \in U$, and a vector $v \in \cH$.
For each $u \in \cH$ the function 
\[
g_{v,u}(z):=\big\langle\tfrac{f(z)-f(z_0)}{z-z_0}v,u\big\rangle
\]
has a removable singularity at $z_0$ since $f$ is WOT holomorphic.
Recalling that $U$ is assumed to be a disk and $g_{v,u}$ extends holomorphically past $\partial U$, it follows that $g_{v,u}(U)$ is bounded as well.
Since this holds for all $u$, and weak boundedness implies boundedness (as a consequence of the principle of uniform boundedness), it follows that $g_v(U)$ is bounded, where $g_v:U \to \cH$ is the function 
\[
g_v(z) = \frac{f(z)-f(z_0)}{z-z_0}v.
\]
Again invoking the principle of uniform boundedness, we conclude that in fact the set of operators $\frac{f(z)-f(z_0)}{z-z_0}$ is bounded in norm.
This means that $\norm{f(z)-f(z_0)} \le M\abs{z-z_0}$ for some constant $M$, and in particular $f$ is (Lipschitz) norm continuous at $z_0$.
As $z_0$ was arbitrary, we conclude that $f$ is norm continuous

As $f$ is continuous, the function $h:U \to \cB(\cH)$ given by
\[
h(z) = \frac{1}{2\pi i}\int_{\partial U} \frac{f(\zeta)}{\zeta-z} \, d\zeta
\]
defines a norm holomorphic function $U \to \cB(\cH)$, and by the single variable Cauchy integral formula
\[
\ip{h(z)v,u} = \frac{1}{2\pi i}\int_{\partial U} \frac{\ip{f(\zeta)v,u}}{\zeta-z} \, d\zeta = \ip{f(z)v,u}
\]
for any $v,u \in \cH$.
Hence $h=f$ and $f$ is in particular norm holomorphic, as required.
\end{proof}

We return to the question of what it means for a representation $\pi:\tAnn_c \to \cB(\cH)$ to be holomorphic.
Recall from Section~\ref{sec: semigroup of annuli} that $\tAnn_c$ is equipped with a complex diffeological structure and a topology.
Recall that a complex diffeological space $X$ is equipped with a family of distinguished maps $M \to X$, called `holomorphic maps,' for each finite-dimensional complex manifold $M$.
If $X$ is equipped with a topology we typically also require that the distinguished holomorphic maps be continuous.
Lemma~\ref{lem: B(H) holomorphicity} equips $\cB(\cH)$
with a natural complex diffeology in the above sense.

A map $X \to Y$ of complex diffeological spaces is called \emph{G\^{a}teau holomorphic} 
if for every holomorphic map $f:M \to X$, the composite $M \overset{f}{\to} X \to Y$ is again holomorphic.
We say that $X \to Y$ is \emph{holomorphic} if it G\^{a}teaux holomorphic and continuous.
For a complex diffeological space $X$, the notion of G\^ateaux holomorphic map $X \to \cB(\cH)$ does not depend on the choice of topology on $\cB(\cH)$ (among the ones listed above), whereas the continuity and hence holomorphicity of the map does depend on the topology. 

The goal of this section is to prove that the representations $\pi:\tAnn_c \to \cB(\cH)$ constructed in Theorem~\ref{thm: representation exists} are holomorphic when $\cB(\cH)$ is equipped with the strong operator topology.
We must show that  i) for every holomorphic map $M \to \tAnn_c$ the composite $M \to \tAnn_c \overset{\pi}{\to} \cB(\cH)$ is again holomorphic, and ii) $\pi$ is continuous when $B(\cH)$ is equipped with the strong operator topology.
We will also show that the restriction of $\pi$ to the subsemigroup of thick annuli $\tAnn_c^\circ$ is holomorphic when $\cB(\cH)$ is equipped with the norm topology.

We begin with a key technical observation.

\begin{lem}\label{lem: locally bounded}
The representation $\pi:\tAnn_c \to \cB(\cH)$ is locally bounded.
That is, for each $A_0 \in \tAnn_c$, there is a neighborhood of $U$ of $A_0$ in $\tAnn_c$ such that $\pi|_U$ is bounded in operator norm.
\end{lem}
\begin{proof}
The projection $\tAnn_c \to \Ann$ is locally trivial by \cite[Prop. 5.5]{HenriquesTener24ax}, which is to say that $A_0$ has a neighborhood of the form $D \times V$ where $D$ is a bounded subset of $\bbC^\times$ and $V$ is a neighborhood of the projected image $[A_0]$ in $\Ann$.
Shrinking $V$ if necessary, by \cite[Prop. 4.20]{HenriquesTener24ax} we may choose a smooth family of framings which exponentiate onto $V$. 
In particular we have a smoothly varying family of paths $X(t,B)$ indexed by $B \in V$ such that $B = \prod_{1 \ge t \ge 0} \Exp(X(t,B))\, dt$.
By further shrinking $V$ we may also assume that $\{X(t,B)\}$ is bounded in $C^2$-norm,
and hence bounded in the Sobolev norm
$\norm{X_{\theta\theta}}_{L^2(S^1)}$ which appears in the right-hand side of \eqref{eqn: norm bound in terms of sobolev}.
For $(z,B) \in U$ we have
\[
\pi(z,B) = z \prod_{1 \ge t \ge 0} \Exp(\pi(X(t,B)))\, dt
\]
and $\pi$ is bounded on $U=D \times V$ by Remark~\ref{rem: explicit norm estimates}.
\end{proof}

\begin{lem}\label{lem: gateaux holomorphic}
The representation $\pi:\tAnn_c \to \cB(\cH)$ is G\^{a}teaux holomorphic.
\end{lem}
\begin{proof}
In light of Lemma~\ref{lem: B(H) holomorphicity}, the conclusion of the lemma is independent of whether we equip $\cB(\cH)$ with the norm topology, SOT, or WOT; we will use the SOT.
Fix an open subset $M \subset \bbC^n$ and a holomorphic family $M \to \tAnn_c$ (denoted $m \mapsto A_m$) which we assume without loss of generality to be embedded in the complex plane.
By definition of the complex diffeological structure on $\tAnn_c$, we may (locally) choose a holomorphic family of framings $h:M \times S^1 \times [0,1] \to \bbC$ and a holomorphic function $z:M \to \bbC^\times$ such that
\[
A_m = z(m)\prod_{1 \ge t \ge 0} \Exp(X(t,m))\,dt
\]
where $X(t,m)$ is the path of inward pointing vector fields associated to the framing $h(m,\cdot,\cdot)$. 
The holomorphicity condition on the framings means that the vector fields $X(t,m)(\theta)$ are jointly smooth in all the variables, and that for each fixed $(t,\theta)$ the function $m \mapsto X(t,m)(\theta)$ is holomorphic.
In particular, this means that the map $m \mapsto X(t,m)$ is holomorphic for the $C^\infty$ topology on the space $\cX(S^1)$ of smooth vector fields on $S^1$.

We must show for each $\xi \in \cH$ that the vectors
\[
\pi(A_m)\xi = z(m)\prod_{1 \ge t \ge 0} \Exp(\pi(X(t,m)))\,dt \, \xi
\]
depend holomorphically on $m$.
We assume without loss of generality that $z \equiv 1$, and begin by considering the case when $\xi \in \cH_2$.
In this case, by Lemma~\ref{eqn: differentiating family of evolution systems} and Remark~\ref{rem: diff in families}, the antiholomorphic derivative in the $i$-th coordinate vanishes:
\begin{align*}
&\quad \partial_{\overline{m_i}} \left(\prod_{1 \ge t \ge 0} \Exp(\pi(X(t,m)))\,dt \right) \xi =\\
&=
\left(\int_0^1 \prod_{1 \ge t \ge x} \Exp(\pi(X(t,m)))\,dt \big[\partial_{\overline{m_i}}\pi(X(t,m)) \big] \prod_{x \ge t \ge 0}  \Exp(\pi(X(t,m)))\, dt \right) \xi\\
& = 0
\end{align*}
because the map $m \mapsto \pi(X(t,m))$ is holomorphic $M \to \cB(\cH_1, \cH_0)$.

Now consider a general $\xi \in \cH$, and choose a sequence $\xi_n \in \cH_2$ such that $\xi_n \to \xi$.
The condition of pointwise-holomorphicity is local, and so by Lemma~\ref{lem: locally bounded} (and the fact that $m \mapsto A_m$ is continuous) we may assume without loss of generality that $\norm{\pi(A_m)} \le K$ for all $m$.
Hence $\norm{\pi(A_m)\xi-\pi(A_m)\xi_n} \le K\norm{\xi - \xi_n}$ and thus $\pi(A_m)\xi_n$ converges to $\pi(A_m)\xi$, uniformly in $m$.
The maps $m \mapsto \pi(A_m)\xi_n$ are holomorphic by our preceding work, and thus $m \mapsto \pi(A_m)\xi$ is holomorphic as well.
\end{proof}

We now complete the main result of the section, establishing continuity of the representation.

\begin{thm}\label{thm: representation is holomorphic}
Let $W$ be a unitary positive energy representation of $\Vir_c$, let $\cH$ be its Hilbert space completion, and let $\pi:\tAnn_c \to \cB(\cH)$ be the representation constructed in Theorem~\ref{thm: representation exists}.
Then $\pi$ is holomorphic when $\cB(\cH)$ is given the strong operator topology, and the restriction of $\pi$ to $\tAnn_c^\circ$ is holomorphic when $\cB(\cH)$ is given the norm topology.
\end{thm}
\begin{proof}
    In light of Lemma~\ref{lem: gateaux holomorphic}, it remains to show that $\pi:\tAnn_c \to \cB(\cH)$ is SOT-continuous, and that the restriction $\pi|_{\tAnn_c^\circ}$ is norm continuous.
    
    We first consider norm continuity.
    The projection $\tAnn_c^\circ \to \Ann^\circ$ is locally trivial by \cite[Prop. 5.5]{HenriquesTener24ax}, and so every centrally extended annulus in $\tAnn_c^\circ$ has a neighborhood of the form $D \times V$ where $V \subset \Ann^\circ$ is open, and $D$ is a bounded subset of $\bbC^\times$.
    The topology and complex diffeology on $\Ann$ are determined by the natural embedding into the Fr\'echet space $C^\infty(S^1) \times C^\infty(S^1)$ (discussed in Section~\ref{sec: semigroup of annuli}), and $\Ann^\circ$ is an open subset of $C^\infty(S^1) \times C^\infty(S^1)$.
    By \cite[Prop. 3.7]{Dineen}, a locally bounded G\^{a}teaux holomorphic function from an open subset of a Fr\'echet space to a Banach space is continuous.
    Recalling from Lemma~\ref{lem: B(H) holomorphicity} that G\^{a}teaux holomorphicity is independent of the topology on $\cB(\cH)$, we conclude by Lemmas~\ref{lem: locally bounded} and \ref{lem: gateaux holomorphic} that $\pi|_{\tAnn_c^\circ}$ is norm continuous, and hence norm holomorphic.

    We now fix $\xi \in \cH$, and we must show that $A \mapsto \pi(A)\xi$ is continuous as a map $\tAnn_c \to \cH$.
    First consider the case when $\xi \in r^{L_0}\cH$ for some fixed $0 < r < 1$, so that $\xi = r^{L_0}\xi'$ for some $\xi'\in\cH$.
    Recall that $r^{\ell_0}$ denotes the standard annulus $1 \ge \abs{z} \ge r$, with its standard boundary parametrizations and lift to $\tAnn_c$, so that $\pi(A)r^{L_0} = \pi(A\circ r^{\ell_0})$.
     The map 
    \[
    \tAnn_c \xrightarrow{-\circ r^{\ell_0}} \tAnn_c^\circ \overset{\pi}{\longrightarrow} \cB(\cH)
    \]
    is continuous, hence the expression
    \[
    \pi(A)\xi = \pi(A)r^{L_0}\xi' = \pi(A \circ r^{\ell_0})\xi'
    \]
    defines a continuous map $\tAnn_c \to \cH$ when $\xi \in r^{L_0}\cH$.

    In the general case, choose a sequence $\xi_n \in r^{L_0}\cH$ such that $\xi_n \to \xi$. 
    By Lemma~\ref{lem: locally bounded} for each fixed $A_0 \in \tAnn_c$ we may choose a neighborhood $V$ such that $\norm{\pi(A)} \le M$ for all $A \in V$, and thus we have
    \[
    \norm{\pi(A)\xi - \pi(A)\xi_n} \le M\norm{\xi - \xi_n} \to 0
    \]
    uniformly on $V$.
    Hence the map $A \mapsto \pi(A)\xi$ is continuous, completing the proof.
\end{proof}

\section{Representations of the Virasoro conformal net}\label{sec:7}

Given an interval $I\subset S^1$ (meaning a closed contractible subset of $S^1$), we have an associated subalgebra $\cX_c(I)\subset \cX_c(S^1)$ of the completed Virasoro algebra \eqref{eq: compl of Virc}.
That Lie algebra integrates to a subsemigroup of `bigons'
\[
\Bigon_c(I)\subset \tAnn_c
\]
defined as the identity connected component of the pullback along the projection map $\tAnn_c\twoheadrightarrow \Ann$ of
$\Bigon(I):=\big\{(A,\varphi_{in},\varphi_{out})\in\Ann\,:\,
\varphi_{in}|_{S^1\setminus I}=\varphi_{out}|_{S^1\setminus I}
\big\}$.
It fits into a central extension
\[
0\to\bbC^\times\to\Bigon_c(I)\to\Bigon(I)\to 0.
\]
An element of $\Bigon(I)$ might look like this:
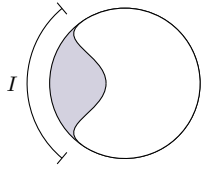
\[
\begin{tikzpicture}[baseline={([yshift=-.5ex]current bounding box.center)}]
	\coordinate (a) at (120:1cm);
	\coordinate (b) at (240:1cm);
	\coordinate (c) at (180:.25cm);
	\fill[fill=red!10!blue!20!gray!30!white] (0,0) circle (1cm);
	\draw (0,0) circle (1cm);
	\fill[fill=white] (a)  .. controls ++(210:.6cm) and ++(90:.4cm) .. (c) .. controls ++(270:.4cm) and ++(150:.6cm) .. (b) -- ([shift=(240:1cm)]0,0) arc (240:480:1cm);
	\draw ([shift=(240:1cm)]0,0) arc (240:480:1cm);
	\draw (a) .. controls ++(210:.6cm) and ++(90:.4cm) .. (c);
	\draw (b) .. controls ++(150:.6cm) and ++(270:.4cm) .. (c);
	\draw (130:1.2cm) -- (130:1.4cm);
	\draw (230:1.2cm) -- (230:1.4cm);
	\draw (130:1.3cm) arc (130:230:1.3cm);
	\node at (180:1.5cm) {\scriptsize{$I$}};
\end{tikzpicture}
\]

Let $(\cH_{c,0},\pi_{c,0})$ denote the vacuum representation of the Virasoro algebra at central charge $c$.
The Virasoro conformal net was defined defined and further studied in \cite{BuchholzSchulz-Mirbach90,Loke,Carpi04,KawahigashiLongo04,Weiner17}.
It assigns to an interval $I$ the von Neumann algebra $\cA_c(I)\subset \cB(\cH_{c,0})$ generated by the unbounded operators $\pi_{c,0}(X)$ for $X\in \cX_c(I)$.\footnote{For more details on how to construct a von Neumann algebra given a collection of closed unbounded operators, see \cite[\S2.2]{CKLW18}.}
Alternatively,
\begin{equation}\label{eqn: diff generates Virasorto net}
\cA_c(I)=\pi_{c,0}\big(\Diff_c(I)\big)'',
\end{equation}
where $\Diff_c(I)\subset \Bigon_c(I)$ is the subgroup that integrates the real subalgebra of $\cX_c(I)$, and fits into a central extension
$0\to U(1)\to\Diff_c(I)\to\Diff(I)\to 0$.

The goal of this section is to construct for every representation $\cH$ of the conformal net $\cA_c$  a holomorphic representation of $\tAnn_c$ that is compatible with the representation of the conformal net.
In order to state the appropriate compatibility condition, we make the following observation.

\begin{lem}\label{lem: bigons inside local algebra}
For $I\subset S^1$ an interval
\begin{equation}\label{eqn: bigons inside local algebra}
\pi_{c,0}(\Bigon_c(I)) \subset \cA_c(I).
\end{equation}
\end{lem}
\begin{proof}
The Haag duality property states that the commutant of $\cA_c(I)$ is given by
$\cA_c(I')=\cA_c(I)'$,
where $I'$ is the complementary interval.
Since $\Bigon_c(I)$ and $\Diff_c(I')$ commute in $\tAnn_c$, their images commute in $\cB(\cH_{c,0})$.
By \eqref{eqn: diff generates Virasorto net} and Haag duality, if follows that
$\pi_{c,0}(\Bigon_c(I)) \subset \cA_c(I)$.
\end{proof}

Now consider an irreducible representation $(\cH_{c,h}, \pi_{c,h})$ of the Virasoro algebra corresponding to a lowest weight $h$ for $L_0$.
By \cite{Weiner17} (see also \cite{Carpi04}), $\cH_{c,h}$ carries a representation $\{\rho_I:\cA_c(I) \to \cB(\cH_{c,h})\}$ of the conformal net $\cA_c$. 
Then in light of Lemma~\ref{lem: bigons inside local algebra}, we have for each interval $I$ a pair of representations of $\Bigon_c(I)$ on $\cH_{c,h}$, given by $\pi_{c,h}$ and $\rho_I \circ \pi_{c,0}$. 
In fact, these representations coincide.

\begin{lem}\pounds{66} \label{lem: compatibility with irreducible net representations}
With notation as above, we have
\[
\rho_I \circ \pi_{c,0}\big|_{\Bigon_c(I)} = \pi_{c,h}\big|_{\Bigon_c(I)}.
\]
\end{lem}
\begin{proof}
Representations of conformal nets (on separable Hilbert spaces) are locally unitarily implementable, which is to say that there exists a unitary $U:\cH_{c,0} \to \cH_{c,h}$ such that $\rho_I(x) = UxU^*$ for all $x \in \cA_c(I)$.
Given $A \in \Bigon_c(I)$, we may write
\[
A = z\cdot \prod_{1 \ge t \ge 0} \Exp(X(t)) \, dt
\]
for $z \in \bbC^\times$, and $X \in \cP^{in}$ with the property that $X(t)$ is supported in $I$ for all $t$.
We then have
\[
\pi_{c,0}(A) = z\prod_{1 \ge t \ge 0} \Exp(\pi_{c,0}(X(t))) \, dt, \qquad \pi_{c,h}(A) = z\prod_{1 \ge t \ge 0} \Exp(\pi_{c,h}(X(t))) \, dt.
\]
Hence
\[
\rho_I(\pi_{c,0}(A)) = U\pi_{c,0}(A)U^* = z \prod_{1 \ge t \ge 0} \Exp(U\pi_{c,0}(X(t))U^*) \, dt.
\]
From the proof of \cite[Thm. 5.6]{Weiner17}, we have $U\pi_{c,0}(X(t))U^* = \pi_{c,h}(X(t))$ as unbounded operators\footnote{%
The identity $U\pi_{c,0}(X)U^* = \pi_{c,h}(X)$ is only claimed in the reference when $X=f \partial_\theta$ with $f$ real-valued.
However, in light of Lemma~\ref{lem: equality unbounded operators} this extends to all $X \in \cX$.
}%
, and so we have $\rho_I(\pi_{c,0}(A)) = \pi_{c,h}(A)$, as required.
\end{proof}

Given an arbitrary representation $\cH$ of the Virasoro conformal net $\cA_c$, we will assign an operator in $\cB(\cH)$ to an element $A \in \tAnn_c$ by factoring $A$ as a product of bigons, and then applying the conformal net representation.
In order to verify that the resulting representation is holomorphic, we need the following technical result that factoring an annulus into bigons can be done holomorphically.

Recall the embedding $\Diff(S^1)\to \Ann:\psi\mapsto A_\psi$ from \eqref{eq: annulus from diffeo}.

\begin{lem}\label{lem: factorizaton of annuli into bigons}\pounds{67}
Let $I_1$ and $I_2$ be two intervals whose interiors cover $S^1$, and let $A_0 \in  \Ann$.
Then there exists $\psi\in\Diff(S^1)$, an open neighborhood $U\subset \Ann$ of $A_0$, and holomorphic maps $f_i:U \to \Bigon(I_i)$ such that the composite
\begin{equation}\label{eq: fact into bigons}
U \xrightarrow{\,\,(f_1,f_2)\,\,} \Bigon(I_1) \times \Bigon(I_2) \xrightarrow{\,\,(A_1,A_2) \,\mapsto\, A_\psi A_1 A_2\,\,} \Ann
\end{equation}
is the inclusion of $U$ into $\Ann$.
\end{lem}

\begin{proof}
Given smooth embeddings $\gamma, \gamma':S^1 \to \bbC^\times$ with winding number $+1$, 
let us write $\gamma \leqslant \gamma'$ to mean that the bounded connected component of $\bbC^\times \setminus \gamma(S^1)$ is contained inside the bounded connected component of $\bbC^\times \setminus \gamma'(S^1)$.
Any pair of such embeddings determines an annulus
\[
A_{\gamma,\gamma'}:=\bbC^\times\backslash\Big( \left[\parbox{3.3cm}{bounded connected \centerline{component of} \centerline{$\bbC^\times \setminus \gamma(S^1)$}}\right]
\cup
\left[\parbox{3.7cm}{unbounded connected \centerline{component of} \centerline{$\bbC^\times \setminus \gamma'(S^1)$}}\right]\Big)
\]
with boundary parametrisations provided by $\gamma$ and $\gamma'$.

Conversely, every annulus $A$ determines a a pair of embeddings $\gamma_{A,in}, \gamma_{A,out}: S^1 \to \bbC^\times$ satisfying $\gamma_{A,in}\leqslant \gamma_{A,out}$ as follows.
Let $\bbD$ be the standard unit disc, equipped with the standard parametrization of its boundary by $S^1$.
And let $\overline\bbD$ denote complex conjugate of $\bbD$, equipped with the boundary parametrization given by $z\mapsto \bar z$.
For each annulus $A$, we identify $\overline \bbD \cup A \cup  \bbD$ with $\mathbb C P^1$ by the unique map which sends $0\in \bbD$ to $0\in\mathbb C P^1$ with unit derivative, and sends $0\in \overline \bbD$ to $\infty\in\mathbb C P^1$.
The boundary parametrizations of $A$ followed by the map $A\hookrightarrow \overline \bbD \cup A \cup  \bbD \to \mathbb C P^1$ determine the two embeddings
\begin{equation}\label{eq: the procedure}
\gamma_{A,in}, \gamma_{A,out} : S^1 \to \bbC^\times.
\end{equation}
These satisfy $\gamma_{A,in}\leqslant  \gamma_{A,out}$.

From now on, we identify $A_0$ with a subset of $\bbC^\times$ following the above procedure.
Its two boundary parametrizations are given by $\varphi_{in} = \gamma_{A_0,in}$ and $\varphi_{out} = \gamma_{A_0,out}$.

Given two intervals $I_1$ and $I_2$ whose interiors cover $S^1$, the intersection $\mathring I_1\cap \mathring I_2$ admits two connected components. Call them $X$ and $Y$.
By replacing $A_0$ by $A_\psi A_0$ for a suitable diffeomorphism $\psi$, we may assume without loss of generality that:
(1) 
either 
$\varphi_{in}|_X= \varphi_{out}|_X$, or there exist a path 
$p:[0,1]\to A_0$ from $\varphi_{in}(X)$ to $\varphi_{out}(X)$ that maps the interior of $[0,1]$ to the interior of $A_0$, and:
(2) either 
$\varphi_{in}|_Y= \varphi_{out}|_Y$, or there exist a path 
$q:[0,1]\to A_0$ from $\varphi_{in}(Y)$ to $\varphi_{out}(Y)$ that maps the interior of $[0,1]$ to the interior of $A_0$.
If $\varphi_{in}|_X= \varphi_{out}|_X$ we define $p:[0,1]\to A_0$ to be any embedding into a proper subset of the interior of $\varphi_{in}(X)$, and similarly for $q$ when $\varphi_{in}|_Y= \varphi_{out}|_Y$.
In all cases,
we can further arrange the paths $p$ and $q$ to be tangent to infinite order to $\partial A_0$ at their endpoints, as depicted in the following picture:
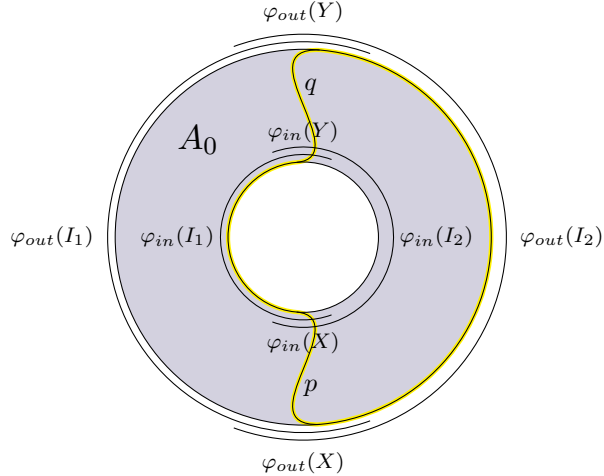
\[
\begin{tikzpicture}
\fill[fill=red!10!blue!20!gray!30!white] circle (2.5);
\fill[fill=white] circle (1);
\draw[ultra thick, yellow]
(95:1) .. controls ++(5:.8cm) and ++(178:1cm) .. (85:2.5)
(-95:1) .. controls ++(-5:.8cm) and ++(-178:1cm) .. (-85:2.5)
(95:1) arc (95:360-95:1)
(85:2.5) arc (85:-85:2.5)
;
\draw circle (2.5);
\draw circle (1);
\draw (70:1.1) arc (70:360-70:1.1);
\draw (110:1.2) arc (110:-110:1.2);
\draw (70:2.6) arc (70:360-70:2.6);
\draw (110:2.7) arc (110:-110:2.7);
\node[left] at (-2.65,0) {$\scriptstyle \varphi_{out}(I_1)$};
\node[right] at (2.75,0) {$\scriptstyle \varphi_{out}(I_2)$};
\node[left] at (-1.05,0) {$\scriptstyle \varphi_{in}(I_1)$};
\node[right] at (1.15,0) {$\scriptstyle \varphi_{in}(I_2)$};
\node at (0,-3) {$\scriptstyle \varphi_{out}(X)$};
\node at (0,3) {$\scriptstyle \varphi_{out}(Y)$};
\node at (0,-1.4) {$\scriptstyle \varphi_{in}(X)$};
\node at (0,1.4) {$\scriptstyle \varphi_{in}(Y)$};
\node[scale=1.2] at (-1.4,1.3) {$A_0$};
\draw (95:1) .. controls ++(5:.8cm) and ++(178:1cm) .. (85:2.5);
\draw (-95:1) .. controls ++(-5:.8cm) and ++(-178:1cm) .. (-85:2.5);
\node[scale=.9] at (.1,-2) {$p$};
\node[scale=.9] at (.1,2) {$q$};
\end{tikzpicture}
\]
Let $S \subset A_0$ be the closed curve, highlighted in yellow in the above picture, obtained by following $\partial_{in}A_0$ for a while, then $p$, then $\partial_{out}A_0$ for a while, then $q$ back to $\partial_{in}A_0$.

The complement of $p([0,1])\cup q([0,1])$ inside $S$ has two connected components.
Let $S_-\subset S$ be the connected component contained in $\varphi_{in}(I_1)$,
and let $S_+\subset S$ be the connected component contained in $\varphi_{out}(I_2)$.
Provided $\psi$, $p$, $q$ are chosen appropriately, the subsets $S^1_-:=\varphi_{in}^{-1}(S_-)$ and $S^1_+:=\varphi_{out}^{-1}(S_+)$ of $S^1$ have disjoint closures.

Our goal is to show that, for $A_0$ as above, there exists an open neighborhood $U$ of $A_0$ in $\Ann$, and holomorphic maps $f_i:U \to \Bigon(I_i)$ ($i=1,2$) such that
\[
U \xrightarrow{\,\,(f_1,f_2)\,\,} \Bigon(I_1) \times \Bigon(I_2) \xrightarrow{\,\,(A_1,A_2) \,\mapsto\, A_1 A_2\,\,} \Ann
\]
is the inclusion of $U$ into $\Ann$
(note the slight difference with \eqref{eq: fact into bigons}).

Let $\delta:S^1\to \bbC^\times$ be an embedding whose image is $S$, that agrees with $\varphi_{in}$ when restricted to $S^1_-$, and agrees with $\varphi_{out}$ when restricted to $S^1_+$.
We can also arrange that $\delta(\theta)-\varphi_{in}(\theta)\,\bot\,\tfrac{d}{d\theta}\varphi_{in}(\theta)$ for all $\theta$ in a small open neighbourhood $(S^1_-)^+$ of the closure of $S^1_-$,
and that $\delta(\theta)-\varphi_{out}(\theta)\,\bot\, \tfrac{d}{d\theta}\varphi_{out}(\theta)$ for all $\theta$ in a small open neighbourhood $(S^1_+)^+$ of the closure of $S^1_+$.
Then 
$A_{\delta,\varphi_{out}}\in \Bigon(I_1)$,
$A_{\varphi_{in},\delta}\in \Bigon(I_2)$,
and 
$A_{\delta,\varphi_{out}}\cup A_{\varphi_{in},\delta}=A_0$.

Pick a partition of unity $\{\lambda_-,\lambda_+,\lambda_\circ\}$ on $S^1$ with
$\lambda_-|_{(S^1_-)^+} \equiv 1$, 
$\lambda_+|_{(S^1_+)+} \equiv 1$, and
$\lambda_\circ = 1-\lambda_--\lambda_+$.
For $A\in \Ann$ (realized as an embedded annulus $A\subset \bbC^\times$ following the procedure described in the paragraph above \eqref{eq: the procedure}), let
\begin{equation}\label{eq: delta}
\delta_A :=
\lambda_- (\gamma_{A,in} + (\delta - \gamma_{A_0,in})) 
+
\lambda_+ (\gamma_{A,out} + (\delta - \gamma_{A_0,out})) 
+
\lambda_\circ\delta.
\end{equation}
Provided $\gamma_{A,in}\leqslant \delta_A\leqslant \gamma_{A,out}$, we have our desired factorization
\[
A_{\delta_A,\gamma_{A,out}}\cup A_{\gamma_{A,in},\delta_A}=A
\]
with $A_{\delta_A,\gamma_{A,out}}\in \Bigon(I_1)$
and $A_{\gamma_{A,in},\delta_A}\in \Bigon(I_2)$.
We finish the proof by noting that the conditions $\delta_A\leqslant \gamma_{A,out}$ and $\gamma_{A,in}\leqslant \delta_A$ are satisfied in an open neighborhood $U$ of $A_0\in \Ann$.\footnote{When $\varphi_{in}|_X= \varphi_{out}|_X$ or
$\varphi_{in}|_Y= \varphi_{out}|_Y$, this uses \cite[Lemma 2.11]{HenriquesTener24ax}.}
Everything in \eqref{eq: delta} depends holomorphically on $A\in U$, so the assignments $A\mapsto A_{\delta_A,\gamma_{A,out}}$ and 
$A\mapsto A_{\gamma_{A,in},\delta_A}$ define holomorphic maps $U\to \Bigon(I_i)$ for $i=1,2$, as desired.
\end{proof}

We are now ready to prove the main result of the section.

\begin{thm} \pounds{72}
Let $(\cH,\rho)$ be a representation of the Virasoro net (not necessarily irreducible). Then there exists a unique holomorphic representation $\pi_\cH$ of $\tAnn_c$ on $\cH$ such that for every interval $I$ we have $\pi_\cH|_{\Bigon_c(I)} = \rho_I \circ \pi_{c,0}|_{\Bigon_c(I)}$.
\end{thm}
\begin{proof}
We first construct a representation $\pi_\cH$ of $\tAnn_c$ on $\cH$ that is compatible with the $\rho_I$.
As a consequence of the split property for $\cA_{c}$ (c.f. \cite{MorinelliTanimotoWeiner16}) and \cite[Prop. 56]{KaLoMu01}), the representation $\rho$ decomposes as a direct integral of irreducible representations.
That is, we have a measurable family of Hilbert spaces $\{\cH_s\}_{s \in S}$, irreducible representations $\rho_s$ of $\cA_c$ on $\cH_s$, and a unitary equivalence
\[
\cH \cong \int^\oplus \cH_s \, ds
\]
which intertwines $\int^\oplus \rho_{s,I} \, ds$ and $\rho_I$ for all $I$.
By \cite[Prop. 2.1]{Carpi04}, every irreducible representation of $\cA_c$ is equivalent to some $\cH_{c,h}$, and so we have representations $\pi_s$ of $\tAnn_c$ on $\cH_s$.
We would like to construct a representation $\pi$ of $\tAnn_c$ on $\cH$ by setting $\pi_\cH = \int^\oplus  \pi_s \, ds$, but in ordered to know that the formula $\int^{\oplus} \pi_s(A)$ defines a bounded operator on $\cH$ we first need to verify the following claim: for every $A \in \tAnn_c$ the family of operators $\pi_s(A)$ is measurable and essentially bounded.
By Lemma~\ref{lem: factorizaton of annuli into bigons} every annulus factors as a product of bigons and a diffeomorphism. In fact, every diffeomorphism factors as a product of diffeomorphisms with support in a proper interval, so in fact every annulus factors as a product of bigons.
The product of measurable essentially bounded families of operators is again a family of the same type, so it suffices to verify the claim when $A \in \Bigon_c(I)$ for some $I$.
In this case, we have
\[
\rho_I(\pi_{c,0}(A)) = \int^\oplus \rho_{s,I}(\pi_{c,0}(A)) \, ds = \int^\oplus \pi_s(A) \, ds
\]
by Lemma~\ref{lem: compatibility with irreducible net representations} (which implicitly asserts that $s\mapsto \pi_s(A)$ is measurable), and since $\rho_I(\pi_{c,0}(A))$ is bounded, the family $\pi_s(A)$ is essentially bounded. 

We now define $\pi_{\cH}:\tAnn_c \to \cB(\cH)$ by the formula 
\[
\pi_\cH(A) = \int^\oplus \pi_s(A) \, ds.
\]
The right-hand side defines a bounded operator by the above claim, and this is evidently a representation of $\tAnn_c$ since each $\pi_s$ is.
Moreover we have already verified that when $A \in \Bigon_c(I)$ we have $\rho_I(\pi_{c,0}(A)) = \pi_\cH(A)$, as required.

It remains to show that $\pi_\cH$ is SOT continuous and Gateaux holomorphic.
We consider continuity first.
Fix $A_0 \in \tAnn_c$.
By \cite[Prop. 5.5]{HenriquesTener24ax}, the projection $\tAnn_c \to \Ann$ is locally trivial, so $A_0$ has a neighborhood of the form $V=D \times U$ where $U$ is a neighborhood of $A_0$ in $\Ann$ and $D$ is a bounded subset of $\bbC^\times$.
By shrinking $U$ if necessary, we may assume by Lemma~\ref{lem: factorizaton of annuli into bigons} that there exist intervals $I_1$ and $I_2$ and holomorphic maps $f_j:U \to \Bigon(I_j)$ such that for every $B \in U$ we have $B = \psi f_1(B)f_2(B)$ for some fixed $\psi \in \Diff(S^1)$.
Further shrinking $U$ if necessary, we may choose continuous lifts $g_j(B) \in \Bigon_c(I_j)$ such that
\[
(z,B) = z \psi g_1(B) g_2(B), \qquad (z,B) \in V.
\]
We then have
\begin{equation}\label{eqn: piH zB factorization}
\pi_\cH(z,B) = z \pi_\cH(\psi) \pi_\cH(g_1(B)) \pi_\cH(g_2(B)).
\end{equation}
Recall that we have $\pi_\cH(g_j(B)) = \rho_{I_j}(\pi_{c,0}(g_j(B)))$ for $j=1,2$.
By Lemma~\ref{lem: locally bounded}, we may further shrink $U$ so that the images $\pi_{c,0}(g_j(U))$ are bounded in norm. 
As representations of von Neumann algebras are norm-non-increasing, it follows that the images $\rho_{I_j}(\pi_{c,0}(g_j(U)) = \pi_{\cH}(g_j(U))$ are bounded in norm as well.

We claim that $B\mapsto \pi_\cH(g_j(B))$  is continuous as a map $U \to \cB(\cH)$, when $\cB(\cH)$ is equipped with the SOT topology.
Indeed, using that $g_j$ is continuous on $U$, that $\pi_{c,0}$ is  SOT-continuous by Theorem~\ref{thm: representation is holomorphic}, and that $\rho_I$ is SOT-SOT continuous on bounded sets (because it's a representation of a von Neumann algebra), it follows that
$B\mapsto \rho_I(\pi_{c,0}(g_j(B)))=\pi_\cH(g_j(B))$ is continuous.
Moreover, as multiplication is SOT-jointly continuous on bounded sets, it follows from the factorization \eqref{eqn: piH zB factorization} that $\pi_\cH(z,B)$ is SOT-continuous as a function of $(z,B)$.

We now show that $\pi_\cH$ is (norm) Gateaux holomorphic.
Let $M \to \tAnn_c:m \mapsto A_m$ be a holomorphic family, for some open subset $M$ of $\bbC^n$.
We fix a point $m_0 \in M$, and assume $m_0=0$ without loss of generality.
Shrinking $M$ if necessary, we may argue as before that $A_0$ has a neighborhood in $\tAnn_c$ of the form $D \times U$, and we may choose $f_j:U \to \Bigon(I_j)$ and $\psi \in \Diff(S^1)$ such that $B=\psi f_1(B) f_2(B)$ for all $B \in U$.
By \cite[Prop. 5.11]{HenriquesTener24ax}, we may choose holomorphic lifts $g_j$ to $\Bigon_c(I_j)$ of the functions $f_j$.
We may thus write 
\[
A_m = z(m) \psi g_1(A_m)g_2(A_m)
\]
for some holomorphic function $z:M \to \bbC^\times$.
Applying $\pi_{\cH}$, this gives:
\begin{equation}\label{eqn: piH Am factorization}
\pi_{\cH}(A_m) = z(m) \pi_{\cH}(\psi) \pi_\cH(g_1(A_m)) \pi_{\cH}(g_2(A_m)).
\end{equation}
Recall that
\[
\pi_\cH(g_j(A_m)) = \rho_{I_j}(\pi_{c,0}(g_j(A_m))).
\]
The maps $m\mapsto g_j(A_m)$ are holomorphic, $\pi_{c,0}$ is norm Gateaux holomorphic by Theorem~\ref{thm: representation is holomorphic}, and $\rho_{I_j}$ is norm holomorphic because it's a bounded linear map.
Hence  $m \mapsto \pi_\cH(g_j(A_m))$ is holomorphic into $\cB(\cH)$ with the norm topology.
As multiplication on $\cB(\cH)$ is norm holomorphic, it now follows from \eqref{eqn: piH Am factorization} that $\pi_{\cH}(A_m)$ is holomorphic as well.
\end{proof}


\begin{thebibliography}{CKLW18}

\bibitem[BSM90]{BuchholzSchulz-Mirbach90}
D.~Buchholz and H.~Schulz-Mirbach.
\newblock Haag duality in conformal quantum field theory.
\newblock {\em Rev. Math. Phys.}, 2(1):105--125, 1990.

\bibitem[Car04]{Carpi04}
S.~Carpi.
\newblock On the representation theory of {V}irasoro nets.
\newblock {\em Comm. Math. Phys.}, 244(2):261--284, 2004.

\bibitem[CKLW18]{CKLW18}
S.~Carpi, Y.~Kawahigashi, R.~Longo, and M.~Weiner.
\newblock From vertex operator algebras to conformal nets and back.
\newblock {\em Mem. Amer. Math. Soc.}, 254(1213), 2018.

\bibitem[CW]{CarpiWeinerLocal}
S.~Carpi and M.~Weiner.
\newblock Local energy bounds and representations of conformal nets.
\newblock {\em In preparation}.

\bibitem[Din99]{Dineen}
S.~Dineen.
\newblock {\em Complex analysis on infinite-dimensional spaces}.
\newblock Springer Monographs in Mathematics. Springer-Verlag London, Ltd.,
  London, 1999.

\bibitem[FH05]{FewsterHollands}
C.~J. Fewster and S.~Hollands.
\newblock Quantum energy inequalities in two-dimensional conformal field
  theory.
\newblock {\em Rev. Math. Phys.}, 17(5):577--612, 2005.

\bibitem[GW85]{GoWa85}
R.~Goodman and N.~R. Wallach.
\newblock Projective unitary positive-energy representations of {${\rm
  Diff}(S\sp 1)$}.
\newblock {\em J. Funct. Anal.}, 63(3):299--321, 1985.

\bibitem[HT24]{HenriquesTener24ax}
A.~G. Henriques and J.~E. Tener.
\newblock The {S}egal-{N}eretin semigroup of annuli.
\newblock {\em arXiv:2410.05929 [math.DG]}, 2024.

\bibitem[KL04]{KawahigashiLongo04}
Y.~Kawahigashi and R.~Longo.
\newblock Classification of local conformal nets. {C}ase {$c<1$}.
\newblock {\em Ann. of Math. (2)}, 160(2):493--522, 2004.

\bibitem[KLM01]{KaLoMu01}
Y.~Kawahigashi, R.~Longo, and M.~M{\"u}ger.
\newblock Multi-interval subfactors and modularity of representations in
  conformal field theory.
\newblock {\em Comm. Math. Phys.}, 219(3):631--669, 2001.

\bibitem[Lok94]{Loke}
T.~Loke.
\newblock {\em Operator algebras and conformal field theory of the discrete
  series representations of {D}iff($S^1$)}.
\newblock PhD thesis, Trinity College, Cambridge, 1994.

\bibitem[LP61]{LumerPhillips}
G.~Lumer and R.~S. Phillips.
\newblock Dissipative operators in a {B}anach space.
\newblock {\em Pacific J. Math.}, 11:679--698, 1961.

\bibitem[MTW16]{MorinelliTanimotoWeiner16}
V.~Morinelli, Y.~Tanimoto, and M.~Weiner.
\newblock Conformal covariance and the split property.
\newblock {\em arXiv:1609.02196 [math-ph]}, 2016.

\bibitem[Ner90]{Neretin90}
Y.~A. Neretin.
\newblock Holomorphic extensions of representations of the group of
  diffeomorphisms of the circle.
\newblock {\em Mathematics of the USSR-Sbornik}, 67(1):75, 1990.

\bibitem[Paz83]{Pazy}
A.~Pazy.
\newblock {\em Semigroups of linear operators and applications to partial
  differential equations}, volume~44 of {\em Applied Mathematical Sciences}.
\newblock Springer-Verlag, New York, 1983.

\bibitem[Seg87]{SegalDef}
G.~Segal.
\newblock The definition of conformal field theory.
\newblock In {\em Topology, geometry and quantum field theory}, volume 308 of
  {\em London Math. Soc. Lecture Note Ser.}, pages 421--577. Cambridge Univ.
  Press, Cambridge, 2004. {O}riginal manuscript from 1987.

\bibitem[TL99]{ToledanoLaredo99}
V.~Toledano~Laredo.
\newblock Integrating unitary representations of infinite-dimensional {L}ie
  groups.
\newblock {\em J. Funct. Anal.}, 161(2):478--508, 1999.

\bibitem[Wei05]{Weiner05}
M.~Weiner.
\newblock {\em Conformal covariance and related properties of chiral {QFT}}.
\newblock PhD thesis, Universit{\`a} di Roma ``Tor Vergata'', 2005.
\newblock arXiv:math/0703336.

\bibitem[Wei17]{Weiner17}
M.~Weiner.
\newblock Local equivalence of representations of {${\rm Diff}^+(S^1)$}
  corresponding to different highest weights.
\newblock {\em Comm. Math. Phys.}, 352(2):759--772, 2017.

\end{thebibliography}
\def\lfhook#1{\setbox0=\hbox{#1}{\ooalign{\hidewidth
  \lower1.5ex\hbox{'}\hidewidth\crcr\unhbox0}}}

\end{document}